\documentclass[ijoo]{informs-ijoo}
\DoubleSpacedXI 

\usepackage{custom_informs}

\usepackage{algorithm}
\usepackage{algorithmic}

\usepackage{booktabs}
\usepackage{enumitem}
\usepackage{subcaption}

\newcommand{\qedwhite}{\hfill \ensuremath{\Box}}
\usepackage{afterpage}
\usepackage{pdflscape}

\usepackage[normalem]{ulem}

\usepackage{tikz}
\usepackage{graphicx,xcolor,colortbl}
\usepackage{epstopdf,epsfig}

\usetikzlibrary{arrows, shapes, shapes.misc}

\tikzstyle{data} = [rectangle, rounded corners, minimum width=1.5cm, minimum height=1.5cm,text width=1.5cm,text centered, draw=black]
\tikzstyle{dataVert} = [rectangle, rounded corners, minimum width=1.5cm, minimum height=7cm,text width=1.5cm,text centered, draw=black, fill=red!20!white]
\tikzstyle{noBox} = [rectangle, rounded corners, minimum width=2cm, minimum height=1.5cm,text width=2cm, text centered]
\tikzstyle{noBoxWide} = [rectangle, rounded corners, minimum width=2cm, minimum height=1.5cm,text width=2cm, text centered]
\tikzstyle{process} = [rectangle, ultra thick, rounded corners, minimum width=2.25cm, minimum height=1.cm, text width=2.25cm, text centered, draw=black,fill=red!20!white]
\tikzstyle{output} = [rectangle, rounded corners, minimum width=1cm, minimum height=.5cm,text width=1cm,text centered, draw=black, fill=red!40!white]
\tikzstyle{arrow} = [ultra thick,->]
\tikzstyle{IOBox} = [rectangle, rounded corners, minimum width=2cm, minimum height=0.5cm,text width=1cm,text centered, draw=black]

\tikzstyle{block} = [rectangle, draw, fill=green!20, thick, text width=2em,align=center, minimum height=2.5em]
\tikzstyle{bigblock} = [rectangle, draw, fill=green!20, thick, text width=6em,align=center, minimum height=4em]

\tikzstyle{every picture}+=[font=\sffamily]

\newcommand{\sspace}{\\\vspace{-1em}}

\newcommand{\GSIO}[1]{\mathbf{GIO}{(\{#1\})}}
\newcommand{\GMIO}[1]{\mathbf{GIO}{(#1)}}

\newcommand{\GIOa}[1]{\mathbf{GIO}_\textnormal{A}{(#1)}}
\newcommand{\GIOap}[1]{\mathbf{GIO}_\textnormal{A}{(#1)}}

\newcommand{\GSIOa}[1]{\mathbf{GIO}_\textnormal{A}{(\{#1\})}}

\newcommand{\GIOr}[1]{\mathbf{GIO}_\textnormal{R}{(#1)}}
\newcommand{\GSIOr}[1]{\mathbf{GIO}_\textnormal{R}{(\{#1\})}}
\newcommand{\GIOd}[1]{\mathbf{GIO}_p{(#1)}}
\newcommand{\GIOp}[2]{\mathbf{GIO}_{#2}{(#1)}}

\newcommand{\GIOrp}[1]{\mathbf{GIO}^+_{\textnormal{R}}{(#1)}}
\newcommand{\GIOrn}[1]{\mathbf{GIO}^-_{\textnormal{R}}{(#1)}}
\newcommand{\GIOrz}[1]{\mathbf{GIO}^0_{\textnormal{R}}{(#1)}}

\newcommand{\GIOrprel}[1]{\mathbf{GIO}^+_{\textnormal{R,LP}}{(#1)}}
\newcommand{\GIOrnrel}[1]{\mathbf{GIO}^-_{\textnormal{R,LP}}{(#1)}}
\newcommand{\GIOrzrel}[1]{\mathbf{GIO}^0_{\textnormal{R,LP}}{(#1)}}

\newcommand{\GIOrpp}[1]{\mathbf{GIO}^+_{\textnormal{R}}{(#1)}}

\newcommand{\GIOK}[2]{\mathbf{GIO}^{(#2)}{(#1)}}

\newcommand{\IORTa}[1]{\mathbf{RT}\text{--}\mathbf{IO}_{\textnormal{A}}{(#1)}}
\newcommand{\IORTr}[1]{\mathbf{RT}\text{--}\mathbf{IO}_{\textnormal{R}}{(#1)}}

\newcommand{\IORTrlp}[1]{\mathbf{RT}\text{--}\mathbf{IO}_{\textnormal{R,LP}}{(#1)}}
\newcommand{\IORTad}[1]{\mathbf{RT}\text{--}\mathbf{IO}_{\textnormal{DN}}{(#1)}}
\newcommand{\IORTadk}[1]{\mathbf{RT}\text{--}\mathbf{IO}_{\textnormal{DN}}{(#1;K)}}
\newcommand{\IORTadlp}[1]{\mathbf{RT}\text{--}\mathbf{IO}_{\textnormal{DN,LP}}{(#1)}}
\newcommand{\FORT}[1]{\mathbf{RT}\text{--}\mathbf{FO}{(#1)}}

\newcommand{\dataset}{\set{\hat{X}}}
\newcommand{\feas}{\set{P}}
\newcommand{\FO}[1]{\mathbf{FO}{(#1)}}
\newcommand{\FOA}{\mathbf{FOA}{(\bc)}}

\newcommand{\proji}[1]{\pi_i{({#1})}}

\newcommand{\projistar}[1]{\pi_{i^*}{({#1})}}

\newcommand{\fproji}[1]{\psi_i{({#1})}}

\newcommand{\fprojistar}[1]{\psi_{i^*}{({#1})}}

\newcommand{\unitvec}[1]{\nu{({#1})}}

\newcommand{\cancelvec}[1]{\mu{({#1})}}

\newcommand{\nnorm}[1]{\left \| #1 \right \|_N}

\newcommand{\rhosp}{\rho(\{\bhx\})}

\newcommand{\LP}{\textnormal{LP}}

\def\tlee#1{\textbf{\color{red}#1}}
\def\dt#1{\textbf{\color{blue}#1}}
\def\rafid#1{\textbf{\color{cyan}#1}}
\def\tcyc#1{\textbf{\color{green}#1}}

    \newcommand{\comm}[1]{#1}

\usepackage{natbib}
 \bibpunct[, ]{(}{)}{,}{a}{}{,}%
 \def\bibfont{\small}%
\TheoremsNumberedThrough     
\ECRepeatTheorems

\EquationsNumberedThrough    

\MANUSCRIPTNO{}

\begin{document}


\RUNAUTHOR{A.~Babier et al.}

\RUNTITLE{An Ensemble Learning Framework for Inverse Linear Optimization}

\TITLE{An Ensemble Learning Framework for Model Fitting and Evaluation in Inverse Linear Optimization}

\ARTICLEAUTHORS{%
        \AUTHOR{Aaron Babier, Timothy C. Y. Chan}
        \AFF{Mechanical \& Industrial Engineering, University of Toronto, Toronto, Canada, \EMAIL{\{ababier, tcychan\}@mie.utoronto.ca}} 
        \AUTHOR{Taewoo Lee}
        \AFF{Industrial Engineering, University of Houston, Houston, Texas, USA, \EMAIL{tlee6@uh.edu}}
        \AUTHOR{Rafid Mahmood}
        \AFF{Mechanical \& Industrial Engineering, University of Toronto, Toronto, Canada, \EMAIL{rafid.mahmood@mail.utoronto.ca}}
        \AUTHOR{Daria Terekhov}
        \AFF{Mechanical, Industrial, \& Aerospace Engineering, Concordia University, Montr\'{e}al, Canada, \EMAIL{daria.terekhov@concordia.ca}}
} 

\ABSTRACT{%
        We develop a generalized inverse optimization framework for fitting the cost vector of a single linear optimization problem given multiple observed decisions. 
        This setting is motivated by ensemble learning, where building consensus from base learners can yield better predictions. We unify several models in the inverse optimization literature under a single framework and derive assumption-free and exact solution methods for each one. We extend a goodness-of-fit metric previously introduced for the problem with a single observed decision to this new setting, and demonstrate several important properties. Finally, we demonstrate our framework in a novel inverse optimization-driven procedure for automated radiation therapy treatment planning. Here, the inverse optimization model leverages an ensemble of dose predictions from different machine learning models to construct a consensus treatment plan that outperforms baseline methods. The consensus plan yields better trade-offs between the competing clinical criteria used for plan evaluation.
}%



\maketitle

%


\section{Introduction}
\label{sec:introduction}

Motivated by the growing availability of data that represent decisions, there is an increasing interest in the use of inverse optimization to gain insight into decision-generating processes and guide subsequent decision-making. 
Inverse optimization has been used in diverse fields, for example capturing equilibrium estimates of asset returns for future portfolio optimization~\citep{ref:bertsimas_or12}, using past electricity market bids to forecast power consumption~\citep{ref:gallego_bidding16}, and estimating incentives to design future health insurance subsidies~\citep{ref:aswani_medicare19}.

Inverse optimization determines optimization model parameters to render a data set of observed decisions minimally sub-optimal for the model. The literature considers different settings that vary based on data characteristics (e.g., a single feasible decision or multiple points from different instances) or the optimization model (e.g., a linear or convex forward problem). A practitioner also chooses a sub-optimality measure to minimize, of which there exist three main variants. The first variant, known as the absolute duality gap, measures the difference between the objective values incurred by data and an imputed optimal value~\citep{ref:bertsimas_mp15, ref:zhao_cdc15, ref:gallego_bidding16, ref:esfahani_oo15}.
The second variant, known as the relative duality gap, measures the ratio instead of the absolute difference~\citep{ref:chan_or14, ref:babier_2018a, ref:chan_gof_16}.
Models using these two measures are referred to as \emph{objective space} models.
The third variant is a \emph{decision space} model that minimizes the distance between observed and optimal decisions~\citep{ref:aswani_arxiv15, ref:esfahani_oo15, ref:aswani_medicare19}.

\comm{
In this paper, we explore an ensemble inverse optimization framework using an arbitrary data set of decisions for a single forward model. Our general motivation is as follows.
Consider a single decision-making problem which we model as a linear program whose cost vector must be estimated. 
Multiple experts generate decisions for the problem. These experts may be human decision makers with their own parameter estimates or even different heuristics applied to the problem; their proposed decisions may be sub-optimal or even infeasible. 
Using these decisions, we impute a single cost vector that best represents the optimization problem attempted by the experts~\citep{ref:troutt_ms95}. We then re-solve the problem with the imputed parameter to generate an optimal decision of similar solution quality to the candidate decisions.
}

Our setting is analogous to ensemble methods in machine learning. Consider the canonical example of a random forest, which averages predictions from a set of decision trees~\citep{ref:breiman_2001}. 
Individual trees train on different subsets of data similar to how individual experts use different experiences to guide their decision making.
An ensemble method averages out the biases of the individual models, just as inverse optimization learns an objective that balances the biases of different decision makers~\citep{ref:troutt_ms95}.
Practical evidence from machine learning shows ensemble methods generally outperform base prediction models.
We similarly show in our application that ensemble inverse optimization can improve over approaches based on individual decisions.

\subsection{Motivating application}

The concrete motivating application in this paper is the automated generation of radiation therapy treatment plans in head-and-neck cancer. Intensity-modulated radiation therapy (IMRT) is one of the most widely-used cancer treatment techniques and is recommended for over 50\% of all cancer cases~\citep{delaney:2005role}. 
\comm{
Because there are multiple competing clinical goals in radiation therapy, multi-objective optimization models are used to design clinically acceptable plans. 
For complex sites such as the head-and-neck, where there may be multiple targets and critical organs, each patient requires a personalized set of objective function weights which are typically obtained via manual parameter tuning. 
This tuning requires going back and forth between a treatment planner and oncologist and may take several days to finalize. 
}
This iterative approach, combined with growing patient volumes, leads to strain in the operation of a cancer center and potential delays in treatment for patients~\citep{Das:2009aa}.


Knowledge-based planning (KBP) is a machine learning-driven planning procedure that automates the design of personalized treatments for each patient, thereby streamlining planning operations~\citep{ref:sharpe_2014}. 
\comm{
KBP contains two components: (i) a prediction model that, for a given patient, predicts an appropriate dose distribution (i.e., a function of the planning decision variables) to deliver; and (ii) an optimization model that generates a deliverable treatment plan that closely replicates the predicted dose. 
While there are various approaches to the optimization stage of KBP, an increasingly popular framework is to use inverse optimization to estimate objective function weights of the original multi-objective planning model by treating the dose predictions as `candidate decisions' obtained from an expert~\citep{ref:chan_or14, ref:babier_2018med}. Re-solving the planning problem with the imputed parameters then yields an optimal treatment plan. 
}

Modern machine learning permits a variety of dose prediction techniques that can predict different representations of the dose~\citep{ref:mcintosh_2016, ref:mahmood_2018gancer, ref:kearney_18}. 
Moreover, treatment plans are clinically evaluated on a set of competing dosimetric criteria and different prediction models lead to plans that over-fit to specific criteria. Given a plethora of prediction models where none are strictly dominating, a naive approach may take each prediction, generate a corresponding treatment plan via optimization, and then compare the plans on their dosimetric performance to identify the best plan for a patient (see Figure~\ref{fig:kbp_pipeline}a). 
However, this approach is excessively laborious and the final plan is still determined from a prediction model that may be over-fit to specific clinical criteria.

\comm{
We propose a natural alternative, which has not been previously considered, to obtain plans that better fit all of the clinical criteria. 
Analogous to an ensemble learning model combining weak predictors to form a better estimate, we harness a set of prediction models into an ensemble inverse optimization model that yields a single optimal treatment plan (see Figure~\ref{fig:kbp_pipeline}b). 
Differences in the prediction models imitate the biases of different clinical experts that may lead them to suggest different plans for a given patient, even though they all aim to satisfy the same clinical criteria. 
Our inverse optimization model is a consensus-building treatment planner whose plans compromise between predictions to satisfy aggregate metrics better than any individual model.
}

\begin{figure}[t]
    \centering
    \caption{
        \comm{An existing KBP pipeline versus our proposed ensemble approach.}
    } 
    \subcaptionbox{
        \comm{Multiple predictors and optimizers. \vspace{-1em} The \\ best plan is identified (highlighted) and used.}
    }[0.49\linewidth]{%
    \resizebox{0.49\linewidth}{!}{%
        \begin{tikzpicture}
            \node(OVH)[noBox] {\includegraphics[width=0.8\textwidth]{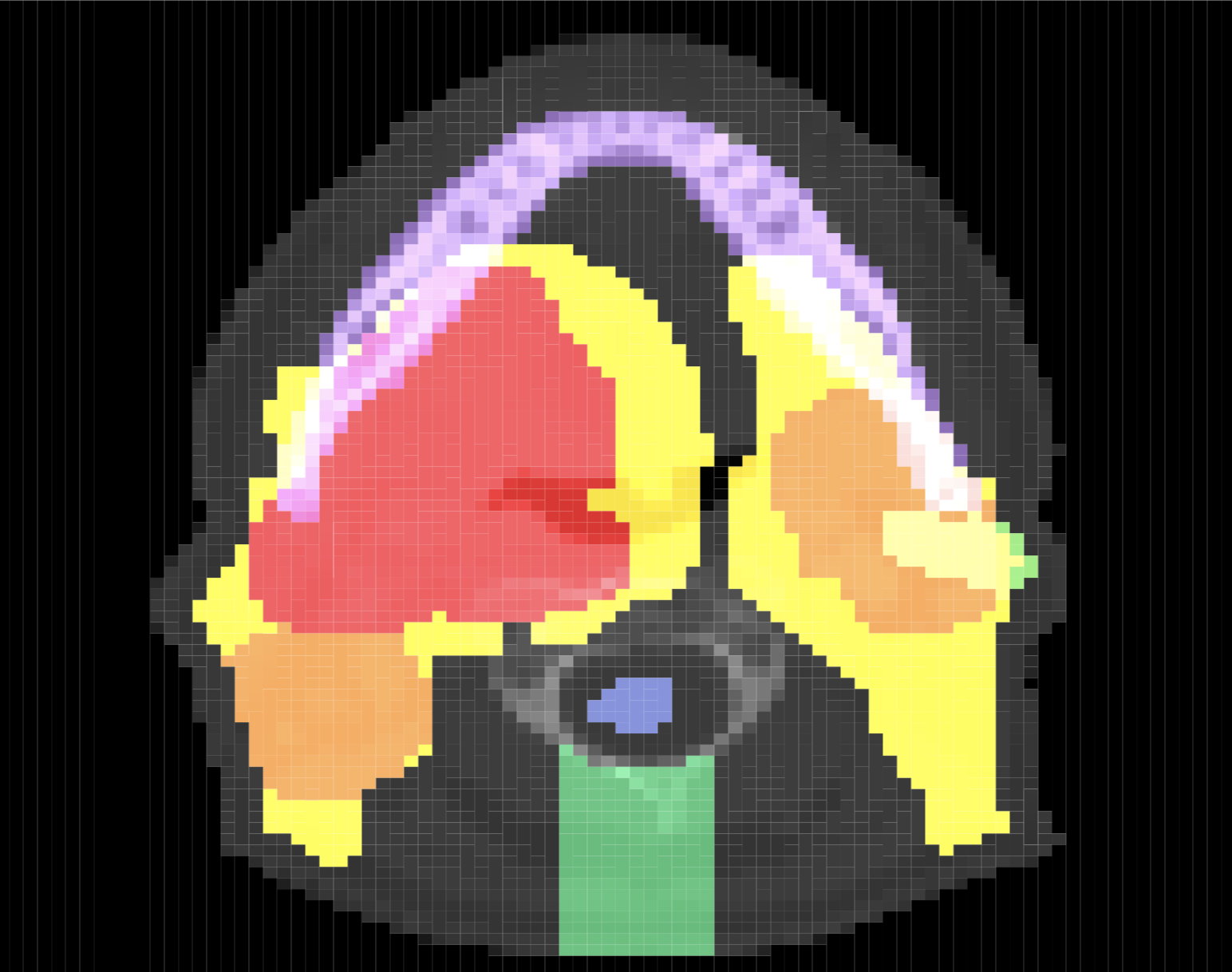}\sspace CT image};
        
            \node(KBP)[process,right of = OVH, xshift=2cm] {Prediction \sspace model B};
            \node(KBPup)[process, above of=KBP, yshift=0.5cm] {Prediction \sspace model A};
    	    \node(KBPdown)[process, below of=KBP, yshift=-0.5cm] {Prediction \sspace model C};
            \node(KBPpred)[noBox, right of = KBP, xshift=2cm] {\includegraphics[width=0.8\textwidth]{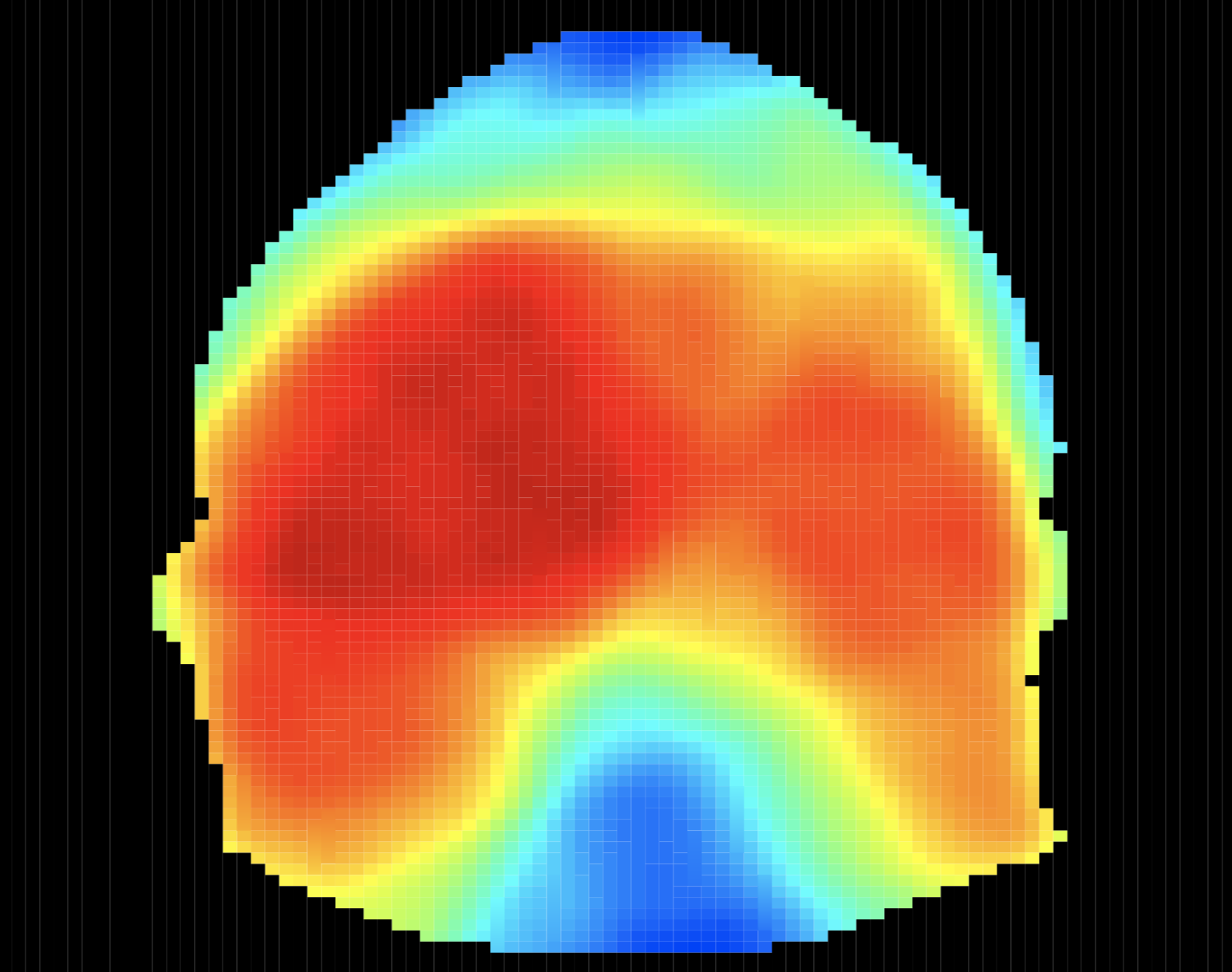}};
    
    		\node(KBPpredup)[noBox, above of = KBPpred, yshift=0.5cm] {\includegraphics[width=0.8\textwidth]{figures/dose.png}};
    		\node(KBPpreddown)[noBox, below of = KBPpred, yshift=-0.5cm] {\includegraphics[width=0.8\textwidth]{figures/dose.png}}; 
    		\node[noBox, below of=KBPpreddown, yshift=-0.25cm] {Multiple \sspace predictions};
    		
    		\node(AP)[process,right of = KBPpred, xshift=2cm] {Inverse \sspace optimization};
    		\node(APup)[process, above of= AP, yshift=0.5cm] {Inverse \sspace optimization};
    		\node(APdown)[process, below of = AP, yshift=-0.5cm] {Inverse \sspace optimization};
    		
            \node(plan)[noBox, right of = AP, xshift=2cm, fill=yellow!80!white] {\includegraphics[width=0.8\textwidth]{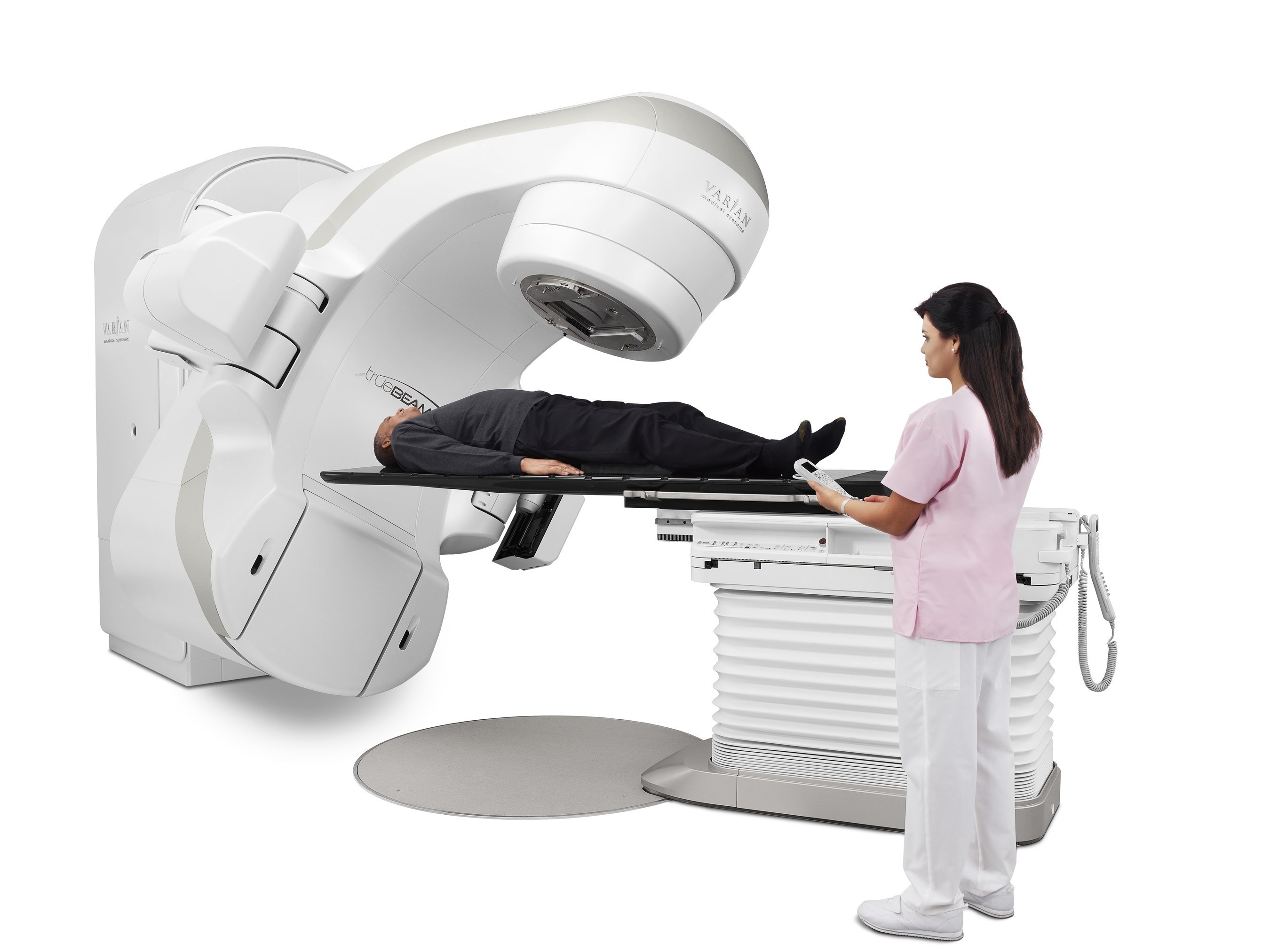}};
            \node(planup)[noBox, above of = plan, yshift=0.5cm] {\includegraphics[width=0.8\textwidth]{figures/LinacWithPpl.jpg}};
            \node(plandown)[noBox, below of = plan, yshift=-0.5cm] {\includegraphics[width=0.8\textwidth]{figures/LinacWithPpl.jpg}};		
    		\node[noBox, below of=plandown, yshift=-0.25cm] {Multiple \sspace plans};

            \draw [arrow] (OVH) -- (KBPup);
            \draw [arrow] (OVH) -- (KBP);
            \draw [arrow] (OVH) -- (KBPdown);
            \draw [arrow] (KBPup) |- (KBPpredup);
            \draw [arrow] (KBP) -- (KBPpred);
            \draw [arrow] (KBPdown) |- (KBPpreddown);
            \draw [arrow] (KBPpredup) -- (APup);
            \draw [arrow] (KBPpreddown) -- (APdown);
            \draw [arrow] (APup) -- (planup);
            \draw [arrow] (APdown) -- (plandown);
            \draw [arrow] (KBPpred) -- (AP);
            \draw [arrow] (AP) |- (plan);
        \end{tikzpicture}
    }
    }
    \subcaptionbox{
        \comm{\vspace{-1em}Multiple predictors are ensembled into one optimization to produce a single plan.}
    }[0.49\linewidth]{%
    \resizebox{0.49\linewidth}{!}{%
        \begin{tikzpicture}
            \node(OVH)[noBox] {\includegraphics[width=0.8\textwidth]{figures/ct.png}\sspace CT image};
        
            \node(KBP)[process,right of = OVH, xshift=2cm] {Prediction \sspace model B};
            \node(KBPup)[process, above of=KBP, yshift=0.5cm] {Prediction \sspace model A};
    	    \node(KBPdown)[process, below of=KBP, yshift=-0.5cm] {Prediction \sspace model C};
            \node(KBPpred)[noBox, right of = KBP, xshift=2cm] {\includegraphics[width=0.8\textwidth]{figures/dose.png}};
    
    		\node(KBPpredup)[noBox, above of = KBPpred, yshift=0.5cm] {\includegraphics[width=0.8\textwidth]{figures/dose.png}};
    		\node(KBPpreddown)[noBox, below of = KBPpred, yshift=-0.5cm] {\includegraphics[width=0.8\textwidth]{figures/dose.png}}; 
    		\node[noBox, below of=KBPpreddown, yshift=-0.25cm] {Multiple \sspace predictions};
    		
    		\node(AP)[process,right of = KBPpred, xshift=2cm] {Ensemble \sspace inverse \sspace optimization};
    		
            \node(plan)[noBox, right of = AP, xshift=2cm] {\includegraphics[width=0.8\textwidth]{figures/LinacWithPpl.jpg}\sspace Ensemble \sspace plan};
    
            \draw [arrow] (OVH) -- (KBPup);
            \draw [arrow] (OVH) -- (KBP);
            \draw [arrow] (OVH) -- (KBPdown);
            \draw [arrow] (KBPup) |- (KBPpredup);
            \draw [arrow] (KBP) -- (KBPpred);
            \draw [arrow] (KBPdown) |- (KBPpreddown);
            \draw [arrow] (KBPpredup) -- (AP);
            \draw [arrow] (KBPpreddown) -- (AP);
            \draw [arrow] (KBPpred) -- (AP);
            \draw [arrow] (AP) |- (plan);
        \end{tikzpicture}
    }
    }
    \label{fig:kbp_pipeline}
\end{figure}

\subsection{Contributions}

Methodologically, we extend the generalized inverse optimization framework of~\citet{ref:chan_gof_16}, which considered only a single feasible decision (``single-point''), to the case of multiple observed decisions (``ensemble'') without any assumption on feasibility. 
\comm{
The previous results do not trivially generalize to our assumption-free settings and we require a suite of new proof techniques to extend the results in these two directions. 
}
Our framework is founded on a flexible model template and specializes to several different models via appropriate specification of model hyperparameters. We develop methods to impute the best-fit cost vector for a variety of different loss measures under a general setting (i.e., no assumptions on data), while also introducing efficient techniques under mild application-specific assumptions. Finally, we generalize a previous goodness-of-fit metric for inverse optimization~\citep{ref:chan_gof_16} to the ensemble case. Together, the model and goodness-of-fit metric form a unified framework for model fitting and evaluation in inverse optimization that is applicable to arbitrary decision data for a single linear optimization problem.

Data-driven inverse optimization has received growing interest, particularly for learning in a class of parametrized convex forward problems~\citep{ref:keshavarz_isic11, ref:bertsimas_mp15, ref:aswani_arxiv15, ref:esfahani_oo15}. Contrasting previous papers, which consider a separate feasible set for each decision, our methods are tailored for a single feasible set, given the motivating assumption of different decision makers solving the same forward problem. Our setting admits more efficient solution algorithms that leverage the geometry of linear programming. 
Furthermore, we develop new bounds relating the performance of different variants that are tighter than bounds adapted from the general convex case to the linear setting~\citep{ref:bertsimas_mp15, ref:esfahani_oo15}.

The specific contributions of our paper are as follows:
\begin{enumerate}
    \item We develop an inverse linear optimization approach applicable to decision data sets of arbitrary size and feasibility for a single optimization problem, motivated by ensemble learning methods. This model is expressed in terms of hyperparameters used to derive different model variants. 
    
    \item We develop exact and assumption-free solution methods for each of the model variants. Under mild data assumptions, we demonstrate how geometric insights from linear optimization can lead to efficient and even analytic solution approaches.
    
    \item We propose a goodness-of-fit metric measuring the model-data fit between a forward problem and arbitrary decision data. We prove several intuitive properties of the metric, including optimality with respect to the inverse optimization model, boundedeness, and monotonicity. 
    
    \item We implement the first ensemble-based automated planning pipeline in radiation therapy, using multiple predictions to design a single treatment for head-and-neck cancer patients. Our plans achieve better clinical trade-offs and our domain-independent goodness-of-fit metric validates our approach.
\end{enumerate}

All proofs can be found in the Electronic Companion.

\section{Background on generalized inverse linear optimization}
\label{sec:preliminaries}

We first review the formulation and main results from~\citet{ref:chan_gof_16}, which introduced an inverse optimization model for linear optimization problems (LPs) unifying both decision and objective space models, but only for a data set with a single feasible observed decision. Let $\bx, \bc \in \field{R}^n$ denote the decision and cost vectors, respectively, and $\bA \in \field{R}^{m \times n}, \bb \in \field{R}^m$ denote the constraint matrix and right-hand side vector, respectively. Let $\set{I} = \left\{1, \dots, m\right\}$ and $\set{J} = \left\{ 1,\dots, n \right\}$. We refer to the following
LP as the forward optimization model 
\begin{align*}
        \FO{\bc}: \quad \minimize{\bx} \quad    & \bc^\tpose \bx \\ 
        \subjectto \quad        & \bx \in \feas := \{\bx \; | \; \bA \bx \geq \bb\}. 
\end{align*}
%
We assume $\feas$ is full-dimensional and $\FO{\bc}$ has no redundant constraints.
Given a \emph{feasible} decision $\bhx \in \feas$, the \emph{single-point} generalized inverse linear optimization problem \citep{ref:chan_gof_16} is
\begin{subequations}
        \label{eq:gio_sp}
        \begin{align}
                \GSIO{\bhx}: \quad \minimize{\bc, \by, \bepsilon} \quad  & \norm{\bepsilon} \\
                \subjectto \quad  & \bA^\tpose \by = \bc, \enskip \by \geq \bzero \label{eq:gio_sp2} \\
                & \bc^\tpose \bhx = \bb^\tpose \by + \bc^\tpose \bepsilon \label{eq:gio_sp3} \\
                & \nnorm{\bc} = 1 \label{eq:gio_sp4} \\
                & \bc \in \set{C}, \bepsilon \in \set{E}. \label{eq:gio_sp6}
        \end{align}
\end{subequations}

Above, $\by \in \field{R}^m$ represents the dual vector for the constraints of the forward problem. Constraints~\eqref{eq:gio_sp2} ensure $\by$ is dual feasible with respect to $\bc$. Constraint~\eqref{eq:gio_sp3} connects $\bc$ and $\by$ with a perturbation vector $\bepsilon \in \field{R}^n$ by enforcing that the pair $(\bhx - \bepsilon, \by)$ satisfy strong duality with respect to $\bc$. Note that these constraints do not imply that the pair is primal-dual optimal (as we have not enforced primal feasibility), but rather that $\bhx-\bepsilon$ lies on a supporting hyperplane $\{ \bx \;|\;\bc^\tpose \bx = \bb^\tpose \by \}$ of the feasible set.
Constraint~\eqref{eq:gio_sp4} is a normalization constraint to prevent the trivial solution of $\bc=\bzero$, where $\nnorm{\cdot}$ denotes an arbitrary norm that may differ from the one in the objective. Finally, constraints~\eqref{eq:gio_sp6} define application-specific perturbation and cost vector restrictions via the sets $\set{E}$ and $\set{C}$, respectively. We leave the choice of the norm in the objective open. The tuple $\left( \norm{\cdot}, \nnorm{\cdot}, \set{C}, \set{E} \right)$ forms the inverse optimization model \emph{hyperparameters}. 
By selecting them appropriately, $\GSIO{\bhx}$ specializes into models that minimize error in objective or decision space. 

%



%
%
%

Although $\GSIO{\bhx}$ is non-convex, it admits a closed-form solution if $\bhx \in \feas$, by finding a projection of $\bhx$ to the boundary of $\feas$ of minimum distance as measured by $\norm{\cdot}$. Specifically, let $\set{H}_i = \left\{ \bx \mid \ba_i^\tpose \bx = b_i \right\}$ be the hyperplane corresponding to the $i^{th}$ constraint and 
\begin{align}
        \proji{\bhx} = \argmin_{\bx \in \set{H}_i}  \norm{\bhx - \bx} \label{eq:def_proji}
\end{align}
be the projection of $\bhx$ to $\set{H}_i$. The hyperplane projection problem has an analytic solution $\proji{\bhx} = \bhx - \frac{\ba_i^\tpose \bhx - b_i}{\normd{\ba_i}} \unitvec{\ba_i}$, where $\normd{\cdot}$ is the dual norm of $\norm{\cdot}$ and $\unitvec{\ba_i} \in \argmax_{\norm{\bv} = 1} \left\{ \bv^\tpose \ba_i \right\}$ \citep{ref:mangasarian_99}. This result leads to an analytic optimal solution to $\mathbf{GIO}(\bhx)$ when $\bhx$ is feasible.
\begin{theorem}[Chan et al., 2019]
        \label{thm:gio_sp_sol}
        Let $\bhx \in \feas$, $i^* \in \argmin_{i \in \set{I}} \left\{ \frac{\ba_i^\tpose \bhx-b_i}{\normd{\ba_i}}\right\}$, and $\BFe_{i}$ be the $i^{th}$ unit vector.
        There exists an optimal solution to $\GSIO{\bhx}$ of the form
        \begin{align}
                \left( \bc^*, \by^*, \bepsilon^* \right) = \left( \frac{\ba_{i^*}}{\nnorm{\ba_{i^*}}}, \frac{\BFe_{i^*}}{\nnorm{\ba_{i^*}}}, \bhx - \projistar{\bhx} \right). \label{eq:gio_sp_sol}
        \end{align}
        %
\end{theorem}
%

If $\bhx \in \feas$, then by Theorem~\ref{thm:gio_sp_sol}, an optimal cost vector describes a supporting hyperplane (i.e., $\left\{ \bx \mid \bc^{*\tpose} \bx = \bb^\tpose \by^* \right\}$) that also corresponds to a constraint of the forward problem.
%


\begin{figure}[t]
        \centering
        \caption{
            \comm{\vspace{-1em}Illustration of Example~\ref{ex:pathological case}. The feasible set for $\FO{\bc}$ is shaded. The black and red squares are $\bhx_q$ and $\bhx_q - \bepsilon^*_q$ from solving $\GSIO{\bhx_q}$ independently. The solid arrows are potential cost vectors.}
        }
        \includegraphics[width=0.343\linewidth]{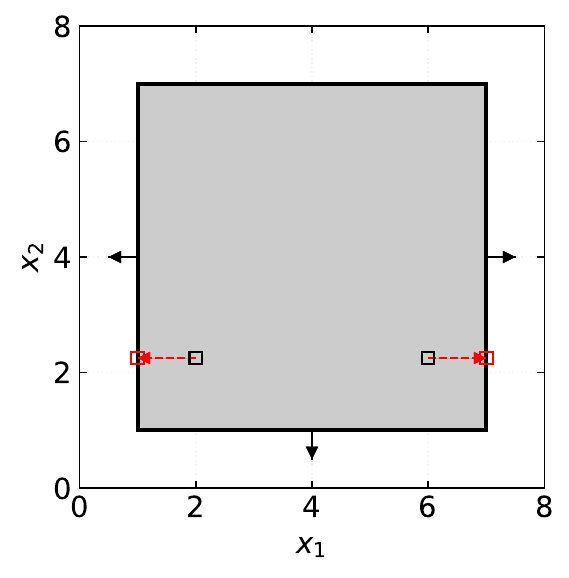}
        \label{fig:pathological}
\end{figure}

\section{Generalized inverse linear optimization with an ensemble of decisions}
\label{sec:multipoint}

We extend $\GSIO{\bhx}$ to the case of multiple observed decisions with no data assumptions. Let $\dataset = \left\{ \bhx_1, \dots, \bhx_Q \right\}$ be a data set, i.e., an ensemble of $Q$ observed decisions, indexed by $\set{Q} = \left\{ 1, \dots, Q \right\}$. We seek to impute a single cost vector $\bc^*$ that minimizes the aggregate loss over all decisions. 

\comm{
Given that Theorem~\ref{thm:gio_sp_sol} admits an analytic solution, one computationally desirable approach may be to solve $\GSIO{\bhx_q}$ for each $\bhx_q$ and impute a set of cost vector estimates. 
We may then consider classical ensemble methods like a random forest, which average weak predictions~\citep{ref:breiman_2001}. However, such a method applied to our setting effectively ignores the geometry of $\feas$, which provides useful information in the estimation of a single cost vector. For example, naively averaging the set of cost vectors to obtain a consensus may lead to pathological outcomes.
}

\comm{
\begin{ex}\label{ex:pathological case}
        Let $\mathbf{FO}(\bc): \underset{\bx}{\min} \{c_1 x_1 + c_2 x_2 \; | \;  x_1 \leq 7, \; x_2 \le 7, \; x_1 \geq 1, \; x_2 \geq 1\}$ and consider two points $\bhx_1 =(2, 2.5)$ and $\bhx_2 = (6, 2.25)$. Solving $\GSIO{\bhx_1}$ and $\GSIO{\bhx_2}$ yields cost vectors $(-1, 0)$ and $(1, 0)$, respectively, with an average cost vector $\bar{\bc} = \bzero$. Note that a more intuitive best-fit cost vector would be $\bc^* = (0, -1)$, pointing to the bottom facet of $\feas$. Figure~\ref{fig:pathological} illustrates this example.
\end{ex}
}

Instead, we design an ensemble inverse optimization model to minimize the aggregate error induced by all points with respect to a single imputed cost vector. We introduce perturbation vectors $\bepsilon_q$ for every $q \in \set{Q}$ and form our problem: 
\begin{subequations}
        \label{eq:gio}
        \begin{align}
                \GMIO{\dataset}: \quad \minimize{\bc, \by, \bepsilon_1,\dots,\bepsilon_Q} \quad    & \sum_{q=1}^Q \norm{\bepsilon_q}\label{eq:gio1} \\ 
                \subjectto \quad    & \bA^\tpose \by = \bc, \enskip \by \geq \bzero \label{eq:gio2} \\
                & \bc^\tpose \bhx_q = \bb^\tpose \by + \bc^\tpose \bepsilon_q, \quad \forall q \in \set{Q} \label{eq:gio3} \\
                & \nnorm{\bc} = 1 \label{eq:gio5} \\
                & \bc \in \set{C}, \bepsilon_q \in \set{E}_q, \quad \forall q \in \set{Q} \label{eq:gio7}.
        \end{align}
\end{subequations}
Constraints~\eqref{eq:gio2} and~\eqref{eq:gio5} are carried from the single-point model, while~\eqref{eq:gio3} and~\eqref{eq:gio7} are ensemble extensions of~\eqref{eq:gio_sp3} and~\eqref{eq:gio_sp6} respectively, ensuring that for each $q \in \set{Q}$, the data points $\bhx_q$ achieve strong duality with respect to $\bc$ after being perturbed by $\bepsilon_q \in \set{E}_q$. The objective minimizes the sum of the norms of the individual perturbation vectors. 
%
%
%
Note that this problem is 
non-convex due to the bilinear terms in~\eqref{eq:gio3} and the normalization constraint~\eqref{eq:gio5}. 
We first show that $\GMIO{\dataset}$ specializes to objective and decision space variants, before developing tailored solution methods.

\subsection{Objective space}
\label{sec:gio_os}
Inverse linear optimization in the objective space is based on the premise that sub-optimal observed decisions are characterized by sub-optimal objective values. Consider the dual problem for $\FO{\bc}$. For each decision $\bhx_q$, the corresponding duality gap is a distance measure
between the objective value of $\bhx_q$ and the optimal value of the dual problem. By choosing the norm in the objective~\eqref{eq:gio1} and the sets $\set{E}_q$ for each $q \in \set{Q}$ appropriately, the problem is transformed to measure a function of the duality gap. We consider two objective space models, the absolute and relative duality gaps.

\subsubsection{Absolute duality gap.}~%
\label{sec:gioadg}
The absolute duality gap method minimizes the aggregate duality gap between the primal objectives of each decision and the imputed dual optimal value:
\begin{subequations}
        \label{eq:gioa}
        \begin{align}
                \GIOa{\dataset}: \quad \minimize{\bc, \by, \epsilon_1,\dots,\epsilon_Q} \quad    & \sum_{q=1}^Q | \epsilon_q | \label{eq:gioa1} \\
                \subjectto \quad    & \bA^\tpose \by = \bc, \enskip \by \geq \bzero \label{eq:gioa2} \\
                & \bc^\tpose \bhx_q = \bb^\tpose \by + \epsilon_q, \quad \forall q \in \set{Q} \label{eq:gioa3} \\
                & \nnorm{\bc} = 1. \label{eq:gioa5}
        \end{align}
\end{subequations}
%
%
This model specializes $\GMIO{\dataset}$ by measuring error in terms of scalar duality gap variables. 
We show that it can be recovered from $\GMIO{\dataset}$ with an appropriate choice of model hyperparameters.
\begin{proposition}
        \label{propn:gioa_specialization}
        Let $\cancelvec{\bc} \in \field{R}^n$ be a parameter satisfying $\norm{\cancelvec{\bc}}_\infty = 1$ and $\cancelvec{\bc}^\tpose \bc = 1$. A solution $\left( \bc^*, \by^*, \epsilon_1^*,\dots,\epsilon_Q^* \right)$ is optimal to $\GIOa{\dataset}$ if and only if $\left( \bc^*, \by^*, \epsilon_1^* \cancelvec{\bc^*}, \dots, \epsilon_Q^* \cancelvec{\bc^*} \right)$ is optimal to $\GMIO{\dataset}$ with  hyperparameters
%
        $ \left( \norm{\cdot}, \nnorm{\cdot}, \set{C}, \set{E}_1, \dots, \set{E}_Q \right) = \left( \norm{\cdot}_\infty, \nnorm{\cdot}, \field{R}^n, \left\{ \epsilon_1 \cancelvec{\bc} \right\}, \dots, \left\{ \epsilon_Q \cancelvec{\bc} \right\} \right)$.
%
\end{proposition}
%
%

 %
Proposition~\ref{propn:gioa_specialization} shows that the specialization of $\GMIO{\dataset}$ to $\GIOa{\dataset}$ depends on each $\bepsilon_q$ being a rescaling of some $\cancelvec{\bc}$ that is dependent only on the cost vector. Note that $\cancelvec{\bc}$ is only a vehicle to aid in the specialization of $\GMIO{\dataset}$, and is useful to interpret solutions of $\GIOa{\dataset}$ in the context of $\GMIO{\dataset}$. For all $\bc$ satisfying $\nnorm{\bc} = 1$, $\cancelvec{\bc}$ must satisfy $\norm{\cancelvec{\bc}}_\infty = 1$ and $\cancelvec{\bc}^\tpose \bc = 1$. Given a specific $\nnorm{\cdot}$, we can propose a structured $\cancelvec{\bc}$. For example, if $\nnorm{\cdot} = \norm{\cdot}_1$, let $\cancelvec{\bc} = \sgn{(\bc)}$ be the sign vector of $\bc$, ensuring that the conditions on $\cancelvec{\bc}$ are satisfied for all $\bc$ with $\norm{\bc}_1 = 1$. If $\nnorm{\cdot} = \norm{\cdot}_\infty$, let $\cancelvec{\bc} = \sgn{(c_{j^*})}\BFe_{j^*}$ be the $j^*$-th unit vector, where $j^* \in \argmax_{j \in \set{J}} \left\{ |c_j| \right\}$. 

\textbf{General solution method.~}
Since the normalization constraint is the sole non-convexity in $\GIOa{\dataset}$, this model can be solved exactly by polyhedral decomposition. The efficiency of this approach depends on the choice of the norm. For example, $2n$ LPs are needed if $\nnorm{\cdot} = \norm{\cdot}_\infty$. 

\begin{theorem}
        \label{thm:gio_reform}
        %
        Let $\left(\bc^*, \by^*, \epsilon_1^*,\dots,\epsilon_Q^*\right)$ be optimal to $\GIOa{\dataset}$ under $\nnorm{\cdot} = \norm{\cdot}_\infty$. There exists $j \in \set{J}$ such that $\left(\bc^*, \by^*, \epsilon_1^*,\dots,\epsilon_Q^*\right)$ is also optimal to $\GIOa{\dataset; j}$, defined as:
        \begin{align}
                \begin{split}
                        \GIOap{\dataset; j}: \quad \minimize{\bc, \by, \epsilon_1,\dots,\epsilon_Q} \quad    & \sum_{q=1}^Q | \epsilon_q | \\
                        \subjectto \quad    & \bA^\tpose \by = \bc, \enskip \by \geq \bzero \\
                        & \bc^\tpose \bhx_q = \bb^\tpose \by + \epsilon_q, \quad \forall q \in \set{Q} \\
                        & \left( c_j = 1 \right) \vee \left( c_j = -1 \right) \\
                        & | c_k | \leq 1, \quad \forall k \in \set{J}/\{j\}.
                \end{split} \label{eq:gioa_reform}
        \end{align}
\end{theorem}
For each $j$, the problem $\GIOap{\dataset; j}$ separates into two LPs (one with the constraint $c_j = 1$ and the other with $c_j = -1$), thus totaling $2n$ LPs. 
\comm{When $\nnorm{\cdot} \neq \norm{\cdot}_\infty$ in general, an exponential number of LPs may be required.} 
We next discuss special cases where this approach simplifies.

\textbf{Non-negative cost vectors.~}%
In many real-world applications, feasible cost vectors should be non-negative (i.e., $\set{C} \subseteq \field{R}^n_+$). Here, it is advantageous to set $\nnorm{\cdot} = \norm{\cdot}_1$, because the normalization constraint becomes $\bc^\tpose \bone = 1$ and $\GIOa{\dataset}$ simplifies to a single linear optimization problem.\@ 

\textbf{Feasible observed decisions.~}%
Most inverse optimization literature focuses on the situation where all observed decisions are feasible for the forward model (i.e., $\dataset \subset \feas$). In this case, $\dataset$ can be replaced by the singleton $\left\{ \bbx \right\}$, where $\bbx$ is the centroid of the points in $\dataset$. A similar result was presented in~\citet[Chapter 4]{ref:goli_thesis}, but for a model with a different normalization constraint that did not prevent trivial solutions. We present the analogous result in the context of our model~\eqref{eq:gioa}. 
\begin{proposition}
        \label{propn:gioa_optimal}
        If $\dataset \subset \feas$ and $\bar{\bx}$ is the centroid of $\dataset$, $\GIOa{\dataset}$ is equivalent to $\GSIOa{\bbx}$.
\end{proposition}
%
%
%
Together, Proposition~\ref{propn:gioa_optimal} and Theorem~\ref{thm:gio_sp_sol} imply that $\GIOa{\dataset}$ is analytically solvable when $\dataset \subset \feas$.

\textbf{Infeasible observed decisions.~}%
Finally, we address scenarios where the observed decisions are all infeasible. We first consider the case where $\dataset$ is a single, infeasible observed decision $\hat\bx$.


\begin{proposition}
        \label{propn:gioa_infeas_sol}
        Assume $\bhx \notin \feas$.
        \begin{enumerate}
                \item If $\bhx$ satisfies $\ba_i^\tpose \bhx > b_i$ for some $i \in \set{I}$, then there also exists $i^* \in \set{I}$ such that $\bty$ is 
                        \begin{align}
                                \label{eq:gioa_sp_infeas_sol}
                                \tilde{y}_i = \frac{1}{\ba_i^\tpose \bhx - b_i}, \quad \tilde{y}_{i^*} = \frac{1}{b_{i^*} - \ba_{i^*}^\tpose \bhx}, \quad \tilde{y}_k = 0 \quad \forall k \in \set{I}\setminus\left\{ i,i^* \right\}
                        \end{align}
                        and $\btc = \bA^\tpose \bty$. The corresponding normalized solution $(\bc^*, \by^*, \epsilon^*) = (\btc / \nnorm{\btc}, \bty / \nnorm{\btc}, 0)$ is an optimal solution to $\GIOa{\{\bhx\}}$ and the optimal value is $0$.

                \item If $\bA \bhx \leq \bb$, there exists $i^* \in \set{I}$ such that~\eqref{eq:gio_sp_sol} is an optimal solution to $\GIOa{\{\bhx\}}$.
        \end{enumerate}
\end{proposition}

\label{ex:gio_adg_infeas_sp}
\begin{figure}[t]
        \centering
        \caption{
            \vspace{-1em}Illustration of Proposition~\ref{propn:gioa_infeas_sol}. \comm{The feasible set for $\FO{\bc}$ is shaded. The black and red squares are $\bhx$ and $\bhx - \bepsilon^*$, respectively. The dashed line is the supporting hyperplane yielding an optimal value of $0$.}
        }
        \subcaptionbox{
            \comm{\vspace{-1em}If one constraint is satisfied, $\bhx$ projects to \\ that constraint.}
            \label{fig:gio_adg_infeas_spA}
        }[0.49\linewidth]{%
                \includegraphics[width=0.343\linewidth]{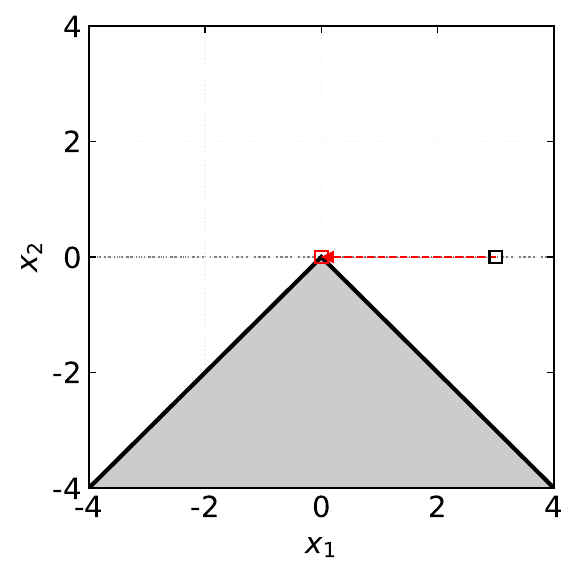}
        }~%
        \subcaptionbox{
            \comm{\vspace{-1em}If all constraints are violated, we solve the inverse problem for $\FOA$ (hatched).}
            \label{fig:gio_adg_infeas_spB}
        }[0.49\linewidth]{%
                \includegraphics[width=0.343\linewidth]{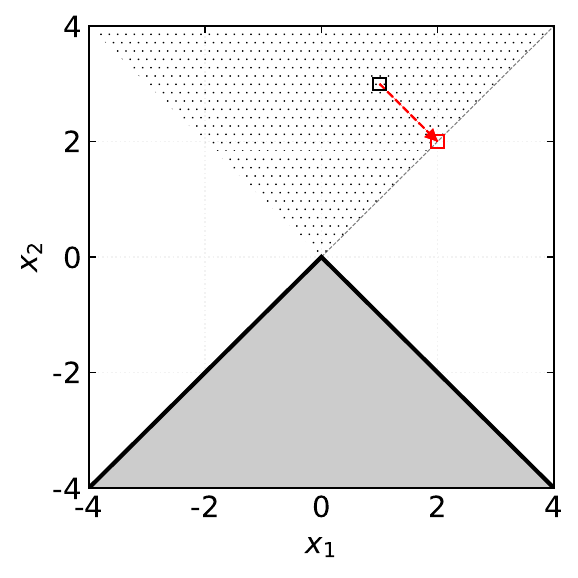}
        }
        \label{fig:gio_adg_infeas_sp}
\end{figure}

Proposition~\ref{propn:gioa_infeas_sol} provides geometric insights regarding the structure of optimal solutions. In objective space inverse optimization, all points that lie on a level set of a cost vector yield the same duality gap. Recall that the hyperplane $\set{H} = \left\{ \bx \mid \bc^{*\tpose} \bx = \bb^\tpose \by^*  \right\}$ is a supporting hyperplane of $\feas$, or in other words, a level set of the cost vector with zero duality gap. If $\bhx \notin \feas$ but satisfies $\ba_i^\tpose \bhx > b_i$ for some $i$, then there always exists a supporting hyperplane that intersects with $\bhx$ (e.g., Figure~\ref{fig:gio_adg_infeas_spA}). If $\bA \bhx \leq \bb$, then no such supporting hyperplane exists. 
However, consider the alternate forward problem $\FOA := \displaystyle\min_\bx \left\{ -\bc^\tpose \bx \mid \bA \bx \leq \bb \right\}$ obtained by reversing the signs of all constraints and the cost vector. The single-point inverse problem for $\bhx$ and $\FOA$ is equivalent to the original problem. Since $\bhx$ is feasible for $\FOA$, Theorem~\ref{thm:gio_sp_sol} applies for $\GIOa{\{\bhx\}}$. Geometrically, the constraints of $\FOA$ describe the nearest supporting hyperplanes of $\FO{\bc}$. Solving one problem solves the other (e.g., Figure~\ref{fig:gio_adg_infeas_spB}, where $\bhx$ projects to an infeasible point for $\FO{\bc}$ with no duality gap). 
\comm{
We use this insight to extend Proposition~\ref{propn:gioa_infeas_sol} Statement 2 for multiple infeasible decisions. If all data points violate all constraints, then the multi-point problem reduces to a single-point.
}


\begin{corollary}
        \label{cor:gioa_infeas_sol_mp}
        Suppose that $\bA \bhx_q \leq \bb$ for all $q \in \set{Q}$, and $\dataset \subset \field{R}^n \setminus \feas$. Let $\bbx$ be the centroid of $\dataset$. Then, $\GIOa{\dataset}$ for the forward problem $\FO{\bc}$ is equivalent to $\GIOa{\{\bbx\}}$ for $\FOA$.
\end{corollary}

\subsubsection{Relative duality gap.}~%
\label{sec:giordg}
The relative duality gap variant minimizes the sum of the ratios between the duality gap for each decision and the imputed dual optimal value for the forward problem:
\begin{subequations}
        \label{eq:gior}
        \begin{align}
                \GIOr{\dataset}: \quad \minimize{\bc, \by, \epsilon_1,\dots,\epsilon_Q} \quad    & \sum_{q=1}^Q \left| \epsilon_q - 1 \right| \label{eq:gior1} \\
                \subjectto \quad    & \bA^\tpose \by = \bc, \enskip \by \geq \bzero \label{eq:gior2} \\
                & \bc^\tpose \bhx_q = \epsilon_q \bb^\tpose \by, \quad \forall q \in \set{Q} \label{eq:gior3} \\
                & \nnorm{\bc} = 1. \label{eq:gior5}
        \end{align}
\end{subequations}
%
%
\noindent Duality gap ratio variables $\epsilon_q$ replace the perturbation vectors used in the general formulation $\GMIO{\dataset}$. These variables are well-defined except when the imputed forward problem has an optimal value of $0$. In this subsection, we assume $\bb \neq \bzero$. 
\comm{
Furthermore, note that if $\bb^\tpose \by = \bzero$ for feasible $\by$, then $\epsilon_q$ are free variables; in this case, we assume $\epsilon_q := 1$ for all $q$. 
}
%
%
First, we show $\GIOr{\dataset}$ can be recovered from $\GMIO{\dataset}$ with appropriate hyperparameters.
\begin{proposition}
        \label{propn:gior_specialization}
        Let $\cancelvec{\bc}$ be a function that satisfies $\norm{\cancelvec{\bc}}_\infty = 1$ and $\cancelvec{\bc}^\tpose \bc = 1$ for all $\bc$. A solution $\left( \bc^*, \by^*, \epsilon_1^*, \dots, \epsilon_Q^* \right)$ for which $\bb^\tpose \by^* \neq 0$, is optimal to $\GIOr{\dataset}$ if and only if $\left( \bc^*, \by^*, \bb^\tpose \by^* \left( \epsilon_1^* - 1 \right)\cancelvec{\bc^*}, \dots, \bb^\tpose \by^* \left( \epsilon_Q^* - 1 \right)\cancelvec{\bc^*} \right)$ is optimal to $\GMIO{\dataset}$ with hyperparameters 
        \begin{align*}
        \left( \norm{\cdot}, \nnorm{\cdot}, \set{C}, \set{E}_1, \dots, \set{E}_Q \right) = \left( \norm{\cdot}_\infty/|\bb^\tpose \by^*|, \nnorm{\cdot}, \field{R}^n, \left\{ \bb^\tpose \by^* \left( \epsilon_1 -1 \right)\cancelvec{\bc^*} \right\}, \dots, \left\{ \bb^\tpose \by^* \left( \epsilon_Q - 1 \right)\cancelvec{\bc^*} \right\} \right).
        \end{align*}
\end{proposition}
%
%

\textbf{General solution method.~}%
Unlike the absolute duality gap problem, which is non-convex only because of the normalization constraint, $\GIOr{\dataset}$ possesses an additional non-convexity due to a bilinear term in the duality gap constraint~\eqref{eq:gior3}. We first address the bilinearity by introducing three sub-problems. We then use polyhedral decomposition to address the normalization constraint.

\begin{proposition}\label{propn:gior_sol}
        Consider the following three problems:\\
        \noindent\begin{minipage}{.32\linewidth}
                \begin{align}
                        \begin{split}
                                &\GIOrp{\dataset; K}: \\
                                \min_{\substack{\bc, \by, \\ \epsilon_1, \dots, \epsilon_Q}} \quad&  \sum_{q=1}^Q | \epsilon_q - 1 | \\
                                \st \quad\quad&  \bA^\tpose \by = \bc, \enskip \by \geq \bzero \\
                                &\bc^\tpose \bhx_q = \epsilon_q, \forall q \in \set{Q} \;\;   \\
                                &\bb^\tpose \by = 1 \\
                                &\nnorm{\bc} \geq K,
                        \end{split} \label{eq:gior_reform_positive}
                \end{align}
        \end{minipage}%
        \begin{minipage}{.35\linewidth}
                \begin{align}
                        \begin{split}
                                &\GIOrn{\dataset; K}: \\
                                \min_{\substack{\bc, \by, \\ \epsilon_1, \dots, \epsilon_Q}} \quad& \sum_{q=1}^Q | \epsilon_q - 1 | \\
                                \st \quad\quad &\bA^\tpose \by = \bc, \enskip \by \geq \bzero \\
                                & \bc^\tpose \bhx_q = -\epsilon_q, \forall q \in \set{Q} \;\;   \\
                                & \bb^\tpose \by = -1 \\
                                & \nnorm{\bc} \geq K,
                        \end{split}\label{eq:gior_reform_negative}
                \end{align}
        \end{minipage}
        \begin{minipage}{.31\linewidth}
                \begin{align}
                        \begin{split}
                                &\GIOrz{\dataset; K}: \\
                                \min_{\bc, \by} \quad& 0 \phantom{\sum_{q=1}^Q | \epsilon_q |}  \\
                                \st \quad& \bA^\tpose \by = \bc, \enskip \by \geq \bzero \\
                                &\bc^\tpose \bhx_q = 0, \forall q \in \set{Q}\;\;  \\
                                &\bb^\tpose \by = 0,  \by^\tpose \bone = 1 \\
                                &\nnorm{\bc} \geq K.
                        \end{split}\label{eq:gior_reform_zero}
                \end{align}
        \end{minipage}
        \hfill \\ \\
        \noindent Let $z^+$ be the optimal value of $\GIOrp{\dataset;K}$ if it is feasible, otherwise $z^+ = \infty$. Let $z^-$ and $z^0$ be defined similarly for $\GIOrn{\dataset;K}$ and $\GIOrz{\dataset;K}$, respectively. Let $z^* = \min\left\{ z^+, z^-, z^0 \right\}$ and let $\left(\bc^*, \by^*,\epsilon_1^*,\dots,\epsilon_Q^*\right)$ be an optimal solution for the corresponding problem. 
        We assume $\epsilon^*_1 = \cdots = \epsilon^*_Q = 1$ for $\GIOrz{\dataset;K}$.  Then there exists $K$ such that the optimal value of $\GIOr{\dataset}$ is equal to $z^*$ and an optimal solution to $\GIOr{\dataset}$ is $\left( \bc^*/\nnorm{\bc^*}, \by^*/\nnorm{\bc^*}, \epsilon_1^*,\dots,\epsilon_Q^* \right)$.
\end{proposition}
\comm{
Proposition~\ref{propn:gior_sol} breaks $\GIOr{\dataset}$ into three cases: $\bb^\tpose \by > 0$, $\bb^\tpose \by < 0$, and $\bb^\tpose \by = 0$. We then normalize $\bb^\tpose \by$ and relax the cost vector normalization constraint~\eqref{eq:gior5}. When $\bb^\tpose \by = \bzero$ for $\GIOrz{\dataset}$, the $\epsilon_q$ terms become free variables, motivating $\epsilon_q = 1$ for all $q$ in order to obtain an objective function $0$.
}

There are two issues to note. First, Proposition~\ref{propn:gior_sol} requires the selection of an appropriate value for the parameter $K$, which can be accomplished by solving an auxiliary problem (see~\ref{sec:ec_gior_K} for details). Second, formulations \eqref{eq:gior_reform_positive}, \eqref{eq:gior_reform_negative}, and \eqref{eq:gior_reform_zero} 
are still non-convex due to the normalization constraint $\nnorm{\bc} \geq K$. As in $\GIOa{\dataset}$, this can be addressed via polyhedral decomposition. 
For example, if $\nnorm{\cdot} = \norm{\cdot}_\infty$, $\GIOrp{\dataset;K}$ decomposes to $2n$ linear programs $\GIOrp{\dataset;K, j}$ (see Theorem~\ref{thm:gio_reform}):
\begin{align}
        \begin{split}
                \GIOrpp{\dataset;K, j}:\quad \minimize{\bc, \by, \epsilon_1, \dots, \epsilon_Q} \quad  & \sum_{q=1}^Q | \epsilon_q - 1 | \\
                \subjectto \quad    & \bA^\tpose \by = \bc, \enskip \by \geq \bzero \\
                & \bc^\tpose \bhx_q = \epsilon_q, \quad \forall q \in \set{Q} \\
                & \bb^\tpose \by = 1 \\
                & \left( c_j \geq K \right) \vee \left( c_j \leq -K \right).
        \end{split} \label{eq:giorp_decompose}
\end{align}
%

\comm{
A complete algorithm for solving $\GIOr{\dataset}$ exactly in this assumption-free setting is provided in \ref{sec:ec_gior_K}. 
We briefly remark here that an alternative approach is to relax the normalization constraints in formulations \eqref{eq:gior_reform_positive}, \eqref{eq:gior_reform_negative}, and \eqref{eq:gior_reform_zero}. If solving the relaxations yields an optimal $\bc^*\neq \bzero$, then this cost vector can be re-scaled as in Proposition~\ref{propn:gior_sol} (see Corollary~\ref{cor:gior_sol_nonzero} in the companion). 
}



\textbf{Feasible observed decisions.~}
%
%
As in the absolute duality gap case, the relative duality gap model reduces to a single-point problem, which has an analytic solution according to Theorem~\ref{thm:gio_sp_sol}. 
\begin{proposition}
        \label{propn:gior_optimal}
        If $\dataset \subset \feas$ and $\bar{\bx}$ is the centroid of $\dataset$, $\GIOr{\dataset}$ is equivalent to $\GSIOr{\bar{\bx}}$.
\end{proposition}
%
%

\textbf{Infeasible observed decisions.~}
%
Proposition~\ref{propn:gior_infeas_sol} is analogous to Proposition~\ref{propn:gioa_infeas_sol}, and yields an analytic solution for $\GIOr{\{\bhx\}}$ if $\bhx \notin \set{P}$. 
\comm{Corollary~\ref{cor:gior_infeas_sol_mp} extends Proposition~\ref{propn:gior_infeas_sol} Statement 2 to multiple decisions similar to Corollary~\ref{cor:gioa_infeas_sol_mp} with Proposition~\ref{propn:gioa_infeas_sol}.} The proofs (omitted) are similar to before. 
\begin{proposition}
        \label{propn:gior_infeas_sol}
        Assume $\bhx \notin \feas$.
        \begin{enumerate}
                \item If $\bhx$ satisfies $\ba_i^\tpose \bhx > b_i$ for some $i \in \set{I}$, then there exists $i^* \in \set{I}$ such that~\eqref{eq:gioa_sp_infeas_sol} is an optimal solution to $\GIOr{\{\bhx\}}$ and the optimal value is 0.

        \item If $\bA \bhx \leq \bb$, there exists $i^* \in \set{I}$ such that~\eqref{eq:gio_sp_sol} is an optimal solution to $\GIOr{\{\bhx\}}$.         \end{enumerate}
\end{proposition}
%

%
\begin{corollary}
        \label{cor:gior_infeas_sol_mp}
        Suppose that $\bA \bhx_q \leq \bb$ for all $q \in \set{Q}$ and let $\bbx$ be the centroid of $\dataset$. 
        Then, $\GIOr{\dataset}$ for the forward problem $\FO{\bc}$ is equivalent to $\GIOr{\{\bbx\}}$ for $\FOA$.
\end{corollary}
%

\subsection{Decision space}

Inverse optimization in the decision space measures error by distance from optimal decisions, rather than objective values. The model identifies a cost vector that produces optimal decisions for the forward problem that are of minimum aggregate distance to the corresponding observed decisions:
%
%
\begin{subequations}
        \label{eq:giod}
        \begin{align}
                \GIOd{\dataset}: \quad \minimize{\bc, \by, \bepsilon_1,\dots,\bepsilon_Q} \quad    & \sum_{q=1}^Q \norm{\bepsilon_q}_p \\
                \subjectto \quad    & \bA^\tpose \by = \bc, \enskip \by \geq \bzero \label{eq:giod2} \\
                & \bc^\tpose \bhx_q = \bb^\tpose \by + \bc^\tpose \bepsilon_q, \quad \forall q \in \set{Q} \label{eq:giod3} \\
                & \bA \left( \bhx_q - \bepsilon_q \right) \geq \bb, \quad \forall q \in \set{Q} \label{eq:giod4} \\
                & \nnorm{\bc} = 1 \label{eq:giod5}.
        \end{align}
\end{subequations}
%
%
$\GIOd{\dataset}$ resembles $\GMIO{\dataset}$, except that the objective function is the sum of $p$-norms ($p \ge 1$) and constraint~\eqref{eq:giod4} is added to enforce primal feasibility of the perturbed decisions $\bhx_q - \bepsilon_q$. 
\comm{
Unlike in the objective space models, we require primal feasibility because the $\bepsilon_q$ perturbation vectors have a physical meaning as the distance from observed $\bhx_q$ to optimal $\bx^*_q$ decisions. 
}
It is straightforward to show $\GIOd{\dataset}$ is a specialization of $\GMIO{\dataset}$ (proof omitted).
\begin{proposition}
        \label{propn:giod_specialization}
        A solution $\left( \bc^*, \by^*, \bepsilon_1^*, \dots, \bepsilon_Q^* \right)$ is optimal to $\GIOd{\dataset}$ if and only if it is optimal to $\GMIO{\dataset}$ with the following hyperparameters:
%
        $\left( \norm{\cdot}, \nnorm{\cdot}, \set{C}, \set{E}_1, \dots, \set{E}_Q \right) = \left( \norm{\cdot}_p, \nnorm{\cdot}, \field{R}^n, \left\{ \bepsilon_1 \mid \bA \left( \bhx_1 - \bepsilon_1 \right) \geq \bb \right\}, \dots, \left\{ \bepsilon_Q \mid \bA \left( \bhx_Q - \bepsilon_Q \right) \geq \bb \right\} \right)$.
%
\end{proposition}
%


Although $\GIOd{\dataset}$ is non-convex, we show that an optimal cost vector coincides with one of the constraints (e.g., Theorem~\ref{thm:gio_sp_sol}).
However, directly projecting all $\bhx_q$ to a hyperplane may result in projections being infeasible, violating~\eqref{eq:giod4}. Thus, we define the \emph{feasible projection problem}:
\begin{align}
        \begin{split}
                \minimize{\bx} \quad    & \norm{\bhx_q - \bx}_p \\
                \subjectto \quad        & \bA \bx \geq \bb \\
                & \ba_i^\tpose \bx = b_i.
        \end{split} \label{eq:def_fproji}
\end{align}
Let $\fproji{\bhx_q}$ be an optimal solution to problem~\eqref{eq:def_fproji}, which identifies the closest point in $\feas$ to $\bhx_q$ on the hyperplane $\set{H}_i = \left\{ \bx \mid \ba_i^\tpose \bx = b_i \right\}$. We first derive a structured optimal solution to $\GIOd{\dataset}$.
\begin{lemma}
        \label{lem:gio_opt}
        There exists $i \in \set{I}$ such that an optimal solution to $\GIOd{\dataset}$ is given by
        \begin{align}
                \left(\bc^*, \by^*, \bepsilon_1^*,\dots,\bepsilon_Q^*\right) = \left( \frac{\ba_i}{\nnorm{\ba_i}}, \frac{\BFe_i}{\nnorm{\ba_i}}, \bhx_1 - \fproji{\bhx_1},\dots, \bhx_Q - \fproji{\bhx_Q} \right) \label{eq:gio_has_optimal}.
        \end{align}
\end{lemma}

The intuition behind Lemma~\ref{lem:gio_opt} is as follows. Given a feasible set of vectors $\bepsilon_1, \dots, \bepsilon_Q$, every observed decision $\bhx_q$ is perturbed by $\bepsilon_q$ to a point that satisfies both strong duality and primal feasibility. Strong duality implies that $\set{H} = \left\{ \bx \mid \bc^{*\tpose} \bx = \bb^\tpose \by^*  \right\}$ is a supporting hyperplane, and so $\bhx_q - \bepsilon_q$ lies on that supporting hyperplane for all $q\in\set{Q}$. 
Every feasible solution not of the form~\eqref{eq:gio_has_optimal} must satisfy multiple constraints with equality, and is dominated by solutions that involve the feasible projection to just one of those constraints. 
Since Lemma~\ref{lem:gio_opt} holds regardless of the chosen norm and feasibility of the observed decisions, we can show $\GIOd{\dataset}$ can be solved via $m$ convex optimization problems (which become linear with appropriate $p$-norms). 
%
%
%
\begin{theorem}
        \label{thm:gio_opt}
        Consider the following optimization problem:
        \begin{subequations}
                \label{eq:gio_i}
                \begin{align}
                        \min_{i \in \set{I}}\quad \min_{\bepsilon_{1,i},\dots, \bepsilon_{Q,i}} \quad & \sum_{q=1}^Q \norm{\bepsilon_{q,i}}_p \label{eq:gio_i1}\\
                        \st \quad\quad    & \bA \left( \bhx_q - \bepsilon_{q,i} \right) \geq \bb, \quad \forall q \in\mathcal{Q} \label{eq:gio_i2}\\
                        & \ba_i^\tpose (\bhx_q - \bepsilon_{q,i}) = b_i \label{eq:gio_i3}, \quad \forall q \in\mathcal{Q}.
                \end{align}
        \end{subequations}
        For each $i \in \set{I}$, let $(\bepsilon_{1,i}^*,\dots,\bepsilon_{Q,i}^*)$ denote an optimal solution to the inner optimization problem and let $i^* \in \argmin_{i \in \set{I}} \sum_{q=1}^Q \norm{\bepsilon_{q,i}^*}$ denote an optimal index determined by the outer optimization problem. Then, $\left( \ba_{i^*}/\nnorm{\ba_{i^*}}, \BFe_{i^*}/\nnorm{\ba_{i^*}}, \bepsilon_{1,i^*}^*,\dots, \bepsilon_{Q,i^*}^* \right)$ is an optimal solution to $\GIOd{\dataset}$.
\end{theorem}
%
%
%

\subsection{Summary of models and comparison with literature}
\label{sec:model_discussion}

Table~\ref{tab:gio_summary} summarizes the model variants. Next, we relate the optimal values of the three variants. 

\begin{theorem}\label{propn:dominance_relationships}
       Assume $\dataset \subset \set{P}$ and let $z_\mathrm{A}^*$ and $z_p^*$ denote the optimal values of $\GIOa{\dataset}$ and $\GIOd{\dataset}$, respectively. Then $z_p^* \geq z_\mathrm{A}^*$.
\end{theorem}
Theorem~\ref{propn:dominance_relationships} implies that if the decision space model returns a low error, so does the absolute duality gap model. 
Note that although bounds between objective and decision space inverse convex optimization models exist (Theorem 1 in~\citet{ref:bertsimas_mp15} and Proposition 2.5 in~\citet{ref:esfahani_oo15}), the previous bounds were developed using constants based on the non-linearity of the objective function of the forward problem (e.g.,~\citet{ref:bertsimas_mp15} assumes the gradient of the objective is strongly monotone), which are not applicable in our linear setting. Furthermore, due to the nature of relative versus absolute measures, we can also bound the performance of the absolute and relative duality gap models, and consequently connect all three variants.

\begin{corollary}\label{cor:dominance_adg_rdg}
        Let $z_\mathrm{A}^*$ and $z_\mathrm{R}^*$ denote the optimal values of $\GIOa{\dataset}$ and $\GIOr{\dataset}$, respectively. Let $f_\mathrm{A}^*$ and $f_\mathrm{R}^*$ be the optimal values of the forward problem $\FO{\bc}$ using cost vectors obtained by $\GIOa{\dataset}$ and $\GIOr{\dataset}$, respectively. Then, $\left| f_\mathrm{R}^* \right|\, z_\mathrm{R}^*  \geq z_\mathrm{A}^* \geq \left| f_\mathrm{A}^* \right|\, z_\mathrm{R}^* $.
\end{corollary}

\begin{table}[t]
        \centering
        \caption{Summary of the different variants of $\GMIO{\dataset}$.}
        \label{tab:gio_summary}
        \begin{tabular}{lccccl}
                \toprule%
                & $\norm{\cdot}$    & $\nnorm{\cdot}$        & $\set{C}$    & $\set{E}_q, \forall q\in\set{Q}$       & Solution approach \\ \midrule
                $\GIOa{\dataset}$      & $\norm{\cdot}_\infty$ & $\nnorm{\cdot}$    & $\field{R}^n$ & $\left\{ \bepsilon_q \mid \bepsilon_q = \epsilon_q \cancelvec{\bc} \right\}$                  & 
                        Polyhedral decomposition
                \\ 
                $\GIOr{\dataset}$      & $\norm{\cdot}_\infty/|\bb^\tpose \by|$     & $\nnorm{\cdot}$     & $\field{R}^n$   & $\left\{ \bepsilon_q \mid \bepsilon_q = \bb^\tpose \by \left( \epsilon_q - 1 \right) \cancelvec{\bc} \right\}$    & Three sub-problems 
                \\ 
                $\GIOd{\dataset}$      & $\norm{\cdot}_p$   & $\nnorm{\cdot}$     & $\field{R}^n$ & $\left\{ \bepsilon_q \mid \bA \left( \bx_q - \bepsilon_q \right) \geq \bb  \right\}$    & \begin{tabular}{@{}l@{}}
                        Formulation~\eqref{eq:gio_i}
                \end{tabular} \\ \bottomrule
        \end{tabular}
\end{table}

Next, we briefly compare our models with similar models from the literature. In-depth technical comparisons are provided in~\ref{sec:ec_inverse_convex}. $\GIOa{\dataset}$ and $\GIOd{\dataset}$ can be seen as special cases of previous inverse convex optimization models~\citep{ref:bertsimas_mp15, ref:aswani_arxiv15, ref:esfahani_oo15}. There, the forward problem is $\min_{\bx} \{ f(\bx;\bu, \bc) \;|\; g(\bx;\bu,\bc) \leq \bzero \}$, where $f(\bx;\bu,\bc)$ and $g(\bx;\bu,\bc)$ are convex differentiable functions and $\bu$ is an exogenous instance-specific parameter. Thus, the data set in their setting is $\dataset = \{(\bhx_1,\bhu_1), \dots, (\bhx_Q, \bhu_Q)\}$. In our setting, we remove $\bu$ and define $f(\bx; \bc) = \bc^\tpose \bx$ and $g(\bx;\bc) = \bb - \bA \bx$ to obtain a linear forward problem with a fixed feasible set. 

While the assumption of instance-specific parameters generalize our setting, we observe that the consequent formulations and methods are on the whole, less efficient than those presented in our paper. Incorporating different forward models, requires additional dual variables and dual feasibility constraints for each feasible set. For a large-scale forward optimization problem, the number of additional variables and constraints required to formulate the inverse problem grows both in the number of feasible sets and the size of $\dataset$. 
\comm{
For example in our application in Section~\ref{sec:imrt}, $n$ (dimension of the decision vector) and $m$ (number of constraints) for the forward problem are on the order of $10^5$. Inverse optimization frameworks from the literature (which impute instance-specific parameters) lead to inverse problems that grow significantly with every data point. In contrast, our ensemble approach using a single forward model does not suffer from this curse.
}


\citet{ref:bertsimas_mp15} study inverse optimization by minimizing a first-order variational inequality (which reduces to the absolute duality gap in LPs) and construct a convex inverse problem without a normalization constraint (e.g., $\nnorm{\bc} = 1$). Although normalization can be avoided with a carefully chosen $\set{C}$, 
setting $f(\bx;\bu,\bc) = \bc^\tpose \bx$ with a general $\set{C} = \field{R}^n$ implies that $(\bc, \by,  \epsilon_1,\dots,\epsilon_Q) = (\bzero, \bzero, 0,\dots,0)$ is a trivially optimal solution.

\citet{ref:esfahani_oo15} study distributionally robust inverse convex optimization problem, which can specialize to absolute duality gap inverse linear optimization with a normalization constraint. Their formulation decomposes to a finite set of conic optimization problems after polyhedral decomposition. While their approach specializes to ours in the non-robust case, we further analyze several other special cases that yield efficient solution methods (e.g.,~Propositions~\ref{propn:gioa_optimal}  and~\ref{propn:gioa_infeas_sol}, and Corollary~\ref{cor:gioa_infeas_sol_mp}). 

\citet{ref:aswani_arxiv15} propose a decision space inverse convex optimization model that satisfies a statistical consistency property given several identifiability conditions that assume the data set of decisions are noisy perturbations of optimal solutions to different forward problems. However, these assumptions may not hold in general, e.g., if they arrive from an ensemble of independent prediction models as in our application (see~\ref{sec:ec_inverse_saa} for details). 
Furthermore, our solution method reformulates $\GIOd{\dataset}$ to $m$ convex problems. In contrast,~\citet{ref:aswani_arxiv15} enumeratively solve the inverse problem using fixed $\bc$ from a quantized subset of $\set{C}$. They state that their algorithm is practical only when the parameter space $\set{C}$ is modest (i.e., at most four or five parameters). 
\comm{
However for $\GIOd{\dataset}$, we assume $\set{C} = \field{R}^n$ and our algorithm for $\GIOd{\dataset}$ is insensitive to $n$.
}

Finally, we remark that the relative duality gap variant has not been studied in inverse convex optimization. It has been studied in inverse linear optimization but only when given a single feasible decision~\citep{ref:chan_gof_16}. Our case study in Section~\ref{sec:imrt} demonstrates the value of $\GIOr{\dataset}$. 

\section{Measuring goodness of fit}
\label{sec:goodness-of-fit}

In this section, we present a unified view of measuring model-data fitness by developing a metric that is easily and consistently interpretable across different inverse linear optimization methods, forward models, and applications. As shown in Example~\ref{ex:gio_rho_good_vs_bad_fit} below, assessing the aggregate error may not provide a complete picture of model fitness, necessitating a context-free fitness metric.

Previously proposed fitness measures for inverse optimization exist but are less general (e.g., \citet{ref:troutt_ms06,ref:chow_12} for application-specific measures or \citet{ref:chan_gof_16} for a metric applicable to only a single feasible decision). Our new metric builds off the latter metric, referred to as the \emph{coefficient of complementarity} and denoted $\rhosp$:
\begin{align*}
        \rhosp = 1 - \frac{\norm{\bepsilon^*}}{\frac{1}{m} \sum_{i=1}^m \norm{\bepsilon_i}}.
\end{align*}
Analogous to the \emph{coefficient of determination} $R^2$ in linear regression, $\rhosp$ provides a scale-free, unitless measure of goodness of fit. The numerator is the residual error from the estimated cost vector, equal to the optimal value of $\GSIO{\bhx}$. The denominator is the average of the errors corresponding to the projections of $\bhx$ to each of the $m$ constraints defining the forward feasible region (i.e., $\bepsilon_i = \bhx - \proji{\bhx}$ for $i \in \set{I}$). Just as $R^2$ calculates the ratio of error of a linear regression model over a baseline mean-only model, $\rhosp$ measures the relative improvement in error from using $\FO{\bc^*}$ compared to a baseline of the average error induced by $m$ candidate cost vectors. 

We now generalize $\rhosp$ for $\GMIO{\dataset}$. For convenience, we omit the data set and denote the absolute duality gap, relative duality gap, and $p$-norm variants of $\rho$ as $\rho_\textnormal{A}$, $\rho_\textnormal{R}$, and $\rho_p$, respectively.

\subsection{Ensemble coefficient of complementarity}

We define the \emph{(ensemble) coefficient of complementarity}, $\rho(\dataset)$, as
\begin{align}
        \label{eq:gof_formula}
        \rho(\dataset) = 1 - \frac{\sum_{q=1}^Q \norm{\bepsilon^*_q}}{\frac{1}{m}\sum_{i=1}^m \left( \sum_{q=1}^Q \norm{\bepsilon_{q,i}} \right) }.
\end{align}
The numerator is the optimal value of $\GMIO{\dataset}$, i.e., the residual error from an optimal solution to the inverse optimization problem. The denominator terms $\sum_{q=1}^Q \norm{\bepsilon_{q,i}}$ represent the aggregate error induced by choosing baseline feasible solutions $(\bc, \by) = (\ba_i/\nnorm{\ba_i}, \BFe_i/\nnorm{\ba_i})$: 
\begin{itemize}
        \item For absolute duality gap, $\GIOa{\dataset}$,
                \begin{align}
                        \label{eq:gof_formula_adg}
                        \sum_{q=1}^Q \norm{\bepsilon_{q,i}} = \sum_{q=1}^Q \frac{\left|\ba_i^\tpose \bhx_q - b_i\right|}{\norm{\ba_i}_1}.
                \end{align}

        \item For relative duality gap, $\GIOr{\dataset}$, under the assumption that $b_i \neq 0$ for all $i \in \set{I}$,
                \begin{align}
                        \label{eq:gof_formula_rdg}
                        \sum_{q=1}^Q \norm{\bepsilon_{q,i}} = \sum_{q=1}^Q \left|\frac{\ba_i^\tpose \bhx_q}{b_i}  -1 \right|.
                \end{align}

        \item For decision space, $\GIOd{\dataset}$, $\sum_{q=1}^Q \norm{\bepsilon_{q,i}}$ are the optimal values of the inner problems in~\eqref{eq:gio_i}.
\end{itemize}

Our choice of baseline (denominator) is a direct extension from the single-point 
case,
where an optimal cost vector can be found by selecting amongst one of the vectors $\ba_i$ defining the $m$ constraints. We maintain this choice of baseline for several reasons. First, an optimal solution will be exactly one of the $\ba_i$ in the general decision space problem (see Lemma~\ref{lem:gio_opt}) and in several special cases of the objective space problem (see Propositions~\ref{propn:gioa_optimal} and~\ref{propn:gior_optimal}). Second, calculation of the denominator is straightforward either directly from the data (e.g.,~\eqref{eq:gof_formula_adg} and~\eqref{eq:gof_formula_rdg}) or via the solution of $m$ convex optimization problems~\eqref{eq:gio_i}.
Third, this definition directly generalizes the single-point metric, 
inheriting several attractive mathematical properties that we present in Section~\ref{subsec:properties}. Finally, given Propositions~\ref{propn:gioa_optimal} and~\ref{propn:gior_optimal}, the ensemble coefficient of complementarity is equal to the single-point version for objective space models when all data points are feasible (proof omitted).
\begin{proposition}
        Let $\bbx$ be the centroid of $\dataset \subset \feas$. Then, $\rho_\textnormal{A}(\dataset) = \rho_\textnormal{A}(\{\bbx\})$ and $\rho_\textnormal{R}(\dataset) = \rho_\textnormal{R}(\{\bbx\})$.
\end{proposition}


\subsection{Properties of $\rho$}\label{subsec:properties}

\begin{theorem}
        \label{thm:rho_properties}
        The following properties hold for $\rho$ defined in~\eqref{eq:gof_formula}:
        \begin{enumerate}
                \item \textbf{Optimality:} $\rho$ is maximized by an optimal solution to $\GMIO{\dataset}$.

                \item \textbf{Boundedness:} $\rho \in [0,1]$.

                \item \textbf{Monotonicity:}
                        For $1 \leq k < n$, let $\GIOK{\dataset}{k}$ be $\GMIO{\dataset}$ with additional constraints
                        $c_i = 0$, for $k+1 \leq i \leq n$ and let $\rho^{(k)}$ be the coefficient of complementarity. Then, $\rho^{(k)} \leq \rho^{(k+1)}$.

        \end{enumerate}
\end{theorem}
These properties are analogous to the properties of $R^2$. The first property underlines how $\rho$ integrates into $\GMIO{\dataset}$. Although one can select any cost vector and calculate the 
$\rho$ value with respect to the data $\dataset$, an optimal cost vector obtained by solving $\GMIO{\dataset}$ is guaranteed to attain the maximum value for $\rho$. Like least squares regression and $R^2$, our inverse optimization model and this $\rho$ metric form a unified framework for model fitting and evaluation in inverse linear optimization.

The second property
makes $\rho$ easily interpretable as a measure of goodness of fit,
with higher values indicating better fit. Note that $\rho=1$ if and only if $\sum_{q=1}^Q \norm{\bepsilon^*_q} = 0$ (i.e., every point in $\dataset$ lies on a supporting hyperplane of $\feas$). In this case, the model perfectly describes all of the data points, analogous to the best fit line passing through all data points in a linear regression. Conversely, $\rho=0$ if and only if  $\sum_{q=1}^Q \norm{\bepsilon^*_q} = \sum_{q=1}^Q \norm{\bepsilon_{q,i}}$ for all $i \in \set{I}$. This scenario occurs when an optimal solution to the inverse optimization problem does not reduce the model-data fit error with respect to any of the baseline solutions, akin to when a linear regression returns an intercept-only model.

The third property states that goodness of fit is nondecreasing as additional degrees of freedom are provided to the modeler, analogous to the property that $R^2$ is nondecreasing in the number of features in a linear regression model. Because of this similarity, $\rho$ also shares a weakness of $R^2$ related to overfitting. When using $\rho$ to compare several models, one should ensure that higher values of $\rho$ represent true improvements in fit, rather than artificial increases that lack generalizability.

\subsection{Numerical examples}
\label{sec:goodness-of-fit_numerical}

\comm{
Examples~\ref{ex:gio_rho_good_vs_bad_fit} and~\ref{ex:gio_rho_scaling} illustrate the value of using $\rho$ instead of an unnormalized error measure such as the aggregate error. Intuitively, a given error with a larger feasible set indicates better fit than the same error in a smaller set. Further, $\rho$ degrades when individual data points are forced to deviate from their preferred cost vector to minimize aggregate error. 
}
Example~\ref{ex:rho_heatmap_fo} showcases $\rho$ for a problem where three points in the data set are fixed and the fourth is varied. Due to primal feasibility in $\GIOp{\dataset}{p}$, decision and objective space yield different $\rho$.

%
%
\begin{figure}[t]
        \centering
        \caption{
            \vspace{-1em}Illustration of Example~\ref{ex:gio_rho_good_vs_bad_fit}. $\GIOa{\dataset}$ with two different $\FO{\bc;u,v}$. \comm{The feasible sets are shaded. The black and red squares are $\bhx_q$ and $\bhx_q - \bepsilon^*_q$, respectively. \vspace{-1em}Both problems yield the same $\bc^*$ and $\bepsilon^*_q$, but have different model fitness $\rho$.}
        }
        \label{fig:gio_rho_good_vs_bad_fit}
        \subcaptionbox{$\FO{\bc;-2,10}$. $\rho = 0.76$.
        \label{fig:rho_good_fit}}[0.49\linewidth]{%
                \includegraphics[width=0.343\linewidth]{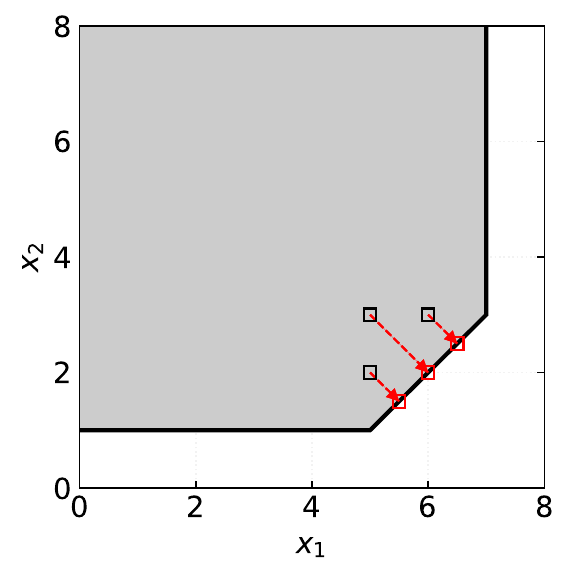}
        }~%
        \subcaptionbox{$\FO{\bc;4,4}$. $\rho = 0.34$.
        \label{fig:rho_bad_fit}}[0.49\linewidth]{%
                \includegraphics[width=0.343\linewidth]{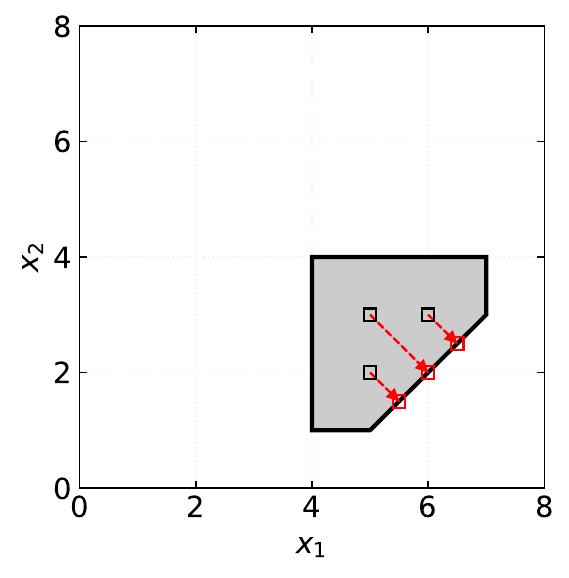}
        }
\end{figure}
\begin{ex}
        \label{ex:gio_rho_good_vs_bad_fit}
        Let $\mathbf{FO}(\bc;u,v): \underset{\bx}{\min}\{c_1 x_1 + c_2 x_2 \; | \; -0.71 x_1 + 0.71 x_2 \geq -2.83, \; x_1 \leq 7, \;  x_2 \leq v, \; x_1 \ge u; \; x_2 \ge 1 \}$ and
        let $\dataset = \left\{ (5,2.5), (4.75,3.75), (5.5,3) \right\}$. \comm{Consider two cases: $\FO{\bc;-2,10}$ and $\FO{\bc;4,4}$.}
        $\GIOa{\dataset}$ yields $\bc^* = \left( -0.5, 0.5 \right)$ and
        $\sum_{q=1}^3 |\epsilon_q^*| = 2.75$ for both, but $\rho=0.76$ for $\FO{\bc;-2,10}$ and $\rho=0.34$ for $\FO{\bc;4,4}$. In Fig.~\ref{fig:rho_good_fit}, $\dataset$ is closer to the bottom facet, relative to the other facets, while in Fig.~\ref{fig:rho_bad_fit}, $\dataset$ is near the ``center'' of the polyhedron rather than one facet.
\end{ex} 

%

%
%
%
\begin{figure}[t]
        \centering
        \caption{
            \vspace{-1em} Illustration of Example~\ref{ex:gio_rho_scaling}. $\GIOa{\dataset}$ with two different data sets for the same $\FO{\bc}$. \comm{The feasible set is shaded. \vspace{-1em}The black and red squares are $\bhx_q$ and $\bhx_q - \bepsilon^*_q$, respectively. Both problems yield the same $\bc^*$ but have different errors $\bepsilon^*_q$ and model fitness $\rho$.}
        }
        \subcaptionbox{$\dataset_1 = \left\{ (3.75, 2), (4, 2.25), (4.25, 2)  \right\}$. $\rho = 0.64$.
        \label{fig:rho_good_scale}}[0.49\linewidth]{%
                \includegraphics[width=0.343\linewidth]{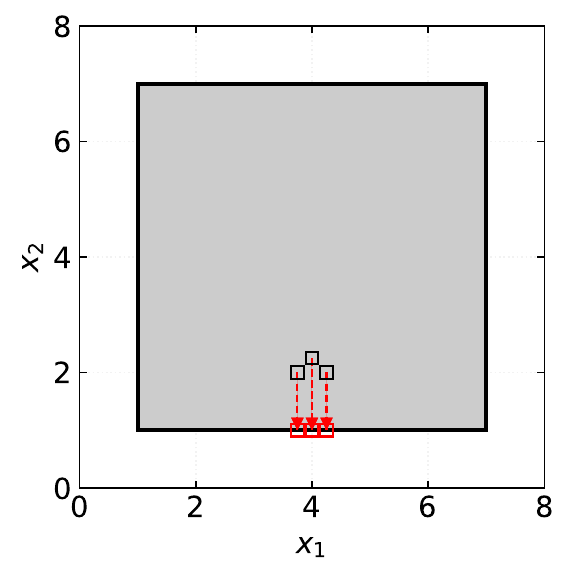}
        }~%
        \subcaptionbox{$\dataset_2 = \left\{ (1.5, 2), (4, 6.25), (6.5, 2)  \right\}$. $\rho = 0.17$.
        \label{fig:rho_bad_scale}}[0.49\linewidth]{%
                \includegraphics[width=0.343\linewidth]{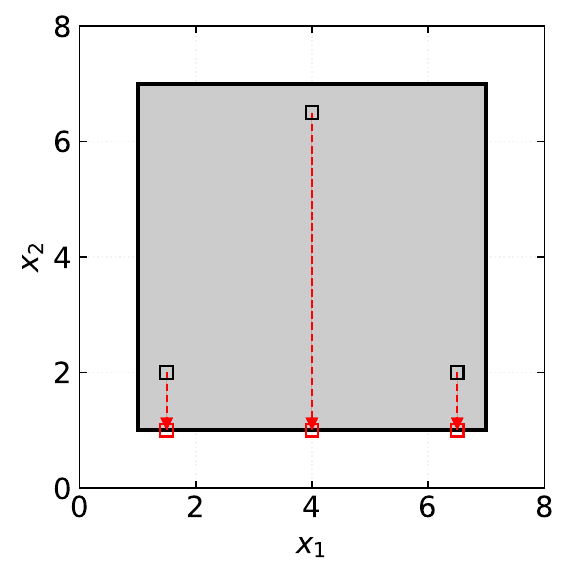}
        }
        \label{fig:gio_rho_good_vs_bad_scale}
\end{figure}
\begin{ex}\label{ex:gio_rho_scaling}
        Let $\mathbf{FO}(\bc): \underset{\bx}{\min} \{c_1 x_1 + c_2 x_2 \; | \;  x_1 \leq 7, \; x_2 \le 7, \; x_1 \geq 1, \; x_2 \geq 1\}$, $\dataset_1 = \left\{ (3.75, 2), (4, 2.25), (4.25, 2) \right\}$ and $\dataset_2 = \left\{ (1.5, 2), (4, 6.25), (6.5, 2) \right\}$.  
        Both $\GIOa{\dataset_1}$ and $\GIOa{\dataset_2}$ impute $\bc^* = \left( 0, 1 \right)$.
        In Fig.~\ref{fig:rho_good_scale}, the points are close together and prefer the bottom facet ($\rho = 0.64$). In Fig.~\ref{fig:rho_bad_scale}, the points are further apart, each with a different preferred cost vector, but aggregate error is minimized by selecting a new different cost vector, resulting in poorer model fit ($\rho= 0.17$).  
\end{ex}

\comm{
Example~\ref{ex:gio_rho_good_vs_bad_fit} and~\ref{ex:gio_rho_scaling} show that the aggregate error and the imputed cost vector from inverse optimization can hide poor model-data fitness. However, poor fitness arises from poor models or poor data. In Example~\ref{ex:gio_rho_good_vs_bad_fit}, the forward model $\FO{\bc;4,4}$ uses constraints that are potentially too tight given the data. On the other hand in Example~\ref{ex:gio_rho_scaling}, the data set $\dataset_2$ is spread out and unlikely to be all generated with respect to a single objective.
}

%
%
%
\begin{figure}[t]
        \centering
        \caption{
            \vspace{-1em}Illustration of Example~\ref{ex:rho_heatmap_fo}. \comm{Heat maps of $\rho$ for different $\GMIO{\dataset}$ where $\dataset$ consists of three fixed points and the fourth variable point. \vspace{-1em}The feasible set is highlighted and the squares are the fixed $\bhx_q$ of $\dataset$. $\rho$ is high for $\GIOa{\dataset}$ along the relevant supporting hyperplanes, but is only high for $\GIOp{\dataset}{2}$ along the facets.}
        }
        \label{fig:gio_rho_heatmaps}
        \subcaptionbox{$\GIOa{\dataset}$
        \label{fig:rho_heatmap_adg}}[0.49\linewidth]{%
                \includegraphics[width=0.343\linewidth]{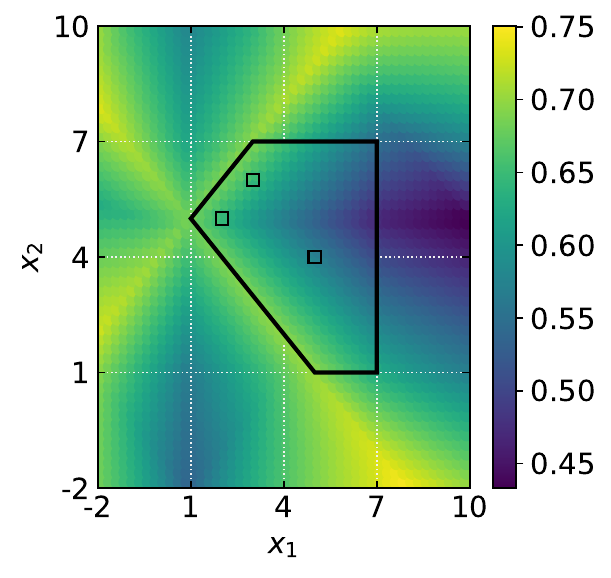}
        }~%
        \subcaptionbox{$\GIOp{\dataset}{2}$
        \label{fig:rho_heatmap_ds}}[0.49\linewidth]{%
                \includegraphics[width=0.343\linewidth]{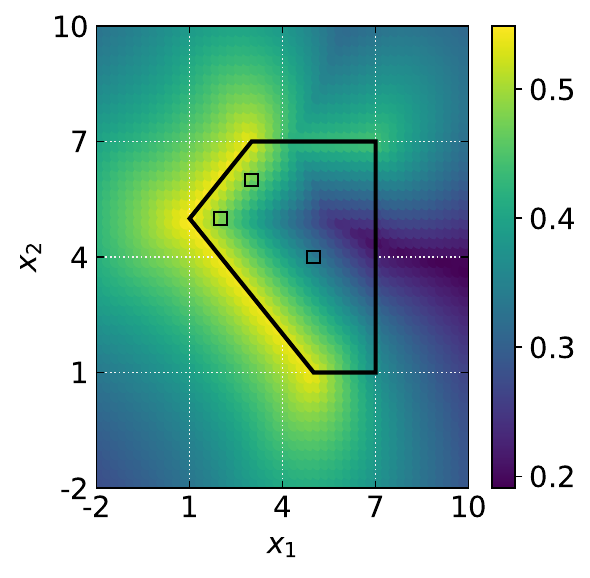}
        }
\end{figure}

\begin{ex}\label{ex:rho_heatmap_fo}
        Let $\mathbf{FO}(\bc): \underset{\bx}{\min} \{c_1 x_1 + c_2 x_2 \; | \;  0.71 x_1 + 0.71 x_2 \geq 4.24, \; 0.71 x_1 -0.71 x_2 \geq -2.83, \; x_1 \leq 7, \; x_2 \leq 7, \; x_2 \ge 1\}$ and consider all data sets of the form $\dataset = \left\{ (2,5), (3,6), (5,4), (\gamma_1, \gamma_2) \right\}$, where $-2 \leq \; \gamma_1, \gamma_2 \; \leq 10$. 
        Fig.~\ref{fig:gio_rho_heatmaps} shows heatmaps of $\rho$ for $\GIOa{\dataset}$ and $\GIOp{\dataset}{2}$.
        \comm{For $\GIOa{\dataset}$, $\rho$ is maximized when the fourth point lies on $\set{H}_1 = \left\{ (x_1,x_2) \mid 0.71 x_1 - 0.71 x_2 = -2.83 \right\}$.} 
        If we solve $\GIOa{\dataset}$ with the three fixed points, then $\bc^* = (0.5, -0.5)$. Thus, when the fourth point lies on $\set{H}_1$, there is zero additional loss. 
        \comm{$\rho$ is also high when the fourth point lies on $\set{H}_2 = \left\{ (x_1, x_2) \mid 0.71 x_1 + 0.71 x_2 = 4.24 \right\}$, and degrades as it moves away from these two hyperplanes.}

        We observe different behavior for $\rho$ in $\GIOp{\dataset}{2}$: maximum model fitness occurs when the fourth point lies along the facets of $\feas$ defined by $\set{H}_1$ and $\set{H}_2$. Due to primal feasibility,  if the fourth point is infeasible, it must project to $\feas$ and thus incur some positive loss. 
\end{ex}
%


\section{Automated knowledge-based planning in radiation therapy}
\label{sec:imrt}

We now implement $\GMIO{\dataset}$ and demonstrate the use of $\rho$ in the context of IMRT treatment planning. In IMRT, a linear accelerator (LINAC) fires beamlets of different intensities that deliver radiation dose to a tumor. A personalized treatment (consisting of beamlet intensity and dose variables) can be designed using a multi-objective optimization model, where the objective weights for a given patient are not known a priori. Knowledge-based planning offers an automated treatment design process. We consider a KBP pipeline where (i) a machine learning model first predicts an appropriate dose distribution for a given patient, (ii) an inverse optimization model treats the dose as an ``observed decision'' to impute candidate objective weights, and (iii) the objective weights are input to the multi-objective planning problem to reconstruct a final plan~\citep{ref:babier_2018a}.

Different prediction models lead to plans that find different trade-offs between clinical evaluation criteria. Rather than a single prediction in KBP, we harness an ensemble of predictions to generate a single treatment plan. 
\comm{However, instead of averaging predictions (like in a random forest), we keep each prediction separate, and feed them all into one inverse optimization model (see Figure~\ref{fig:kbp_pipeline}b).} Until now, KBP has never been used to generate a single plan from multiple predictions. 

We develop an ensemble KBP approach using eight different predictions and show that the relative duality gap model dominates the absolute duality gap model for this application. Plans from the relative duality gap model outperform most single-point models on our overall clinical metric. 
Finally by removing certain low-quality predictions, we design a final model that outperforms the single-point as well as conventional ensemble baselines. Although the final model requires clinically-driven model engineering, we use $\rho$ as domain-independent validation of the clinical intuition.


\subsection{Data and methods}

We use a data set of 217 clinical treatment plans for patients with oropharyngeal (a subset of head-and-neck) cancer, randomly split into 130 plans for training and 87 plans for testing. 
\comm{The training set is used only to pre-train eight prediction models; the test set is used to implement our KBP pipeline.} 
With each patient $k$, we associate parameters $\left( \bC_k, \bA_k, \bb_k \right)$ to a multi-objective linear optimization problem $\FORT{\balpha_k}: \min_{\bx}\left\{ \balpha_k^\tpose \bC_k \bx \mid \bA_k \bx \geq \bb_k, \bx \geq \bzero\right\}$, where $\bC_k$ is a matrix whose rows represent different cost vectors and $\balpha_k$ is a vector of objective weights. The decision vector contains two subvectors, $\bx = (\bw, \bd)$, where $\bw$ is the intensity of each radiation beamlet and $\bd$ is the dose delivered to each voxel (4 mm $\times$ 4 mm $\times$ 2 mm volumetric pixel) of the patient's body, computed by a linear transformation of $\bw$. This multi-objective model fits into $\GMIO{\dataset_k}$ by specifying the set of feasible cost vectors for patient $k$ as $\set{C}_k = \left\{ \bC_k^\tpose \balpha \mid \balpha \geq \bzero \right\}$. Note that the optimization problem for each patient is distinct. For a specific patient, the feasible set is fixed and a single treatment optimization problem is solved. The ensemble arises from the multiple dose predictions for the patient. 
\comm{
Furthermore, our dose predictions are function outputs of the decisions and sufficient for use in inverse optimization. The prediction and optimization models are detailed in \ref{sec:ec_imrt_gio} and summarized below.
}

We first train four different dose prediction models from the literature, labeled Random Forest (RF), 2-D RGB GAN, 2-D GANCER, and 3-D GANCER 
\citep{ref:babier_2018a, ref:babier_2018med, ref:mahmood_2018gancer}. 
For each model, we also implement versions with scaled predictions (suffixed with `-sc.'), which are known to produce plans that better satisfy target (tumor) criteria~\citep{ref:babier_2018med}. Thus, we have eight predictions per patient, which vary in their dose trade-offs between the targets and healthy organs. We predict the dose $\bhd_{k,q}$ for each test patient $k\in\{1,\dots,87\}$ with prediction model $q\in\{1,\dots,8\}$ and let $\dataset_k = \left\{ \bhd_{k,1},\dots,\bhd_{k,8} \right\}$ be data for each patient-specific problem. We then use inverse optimization to construct an optimal treatment plan given these predictions. 

For each patient $k$ in the test set, we implement the absolute and relative duality gap models, referred to as $\IORTa{\dataset_k}$ and $\IORTr{\dataset_k}$, respectively. They are derived from $\GIOa{\dataset_k}$ and $\GIOr{\dataset_k}$ by setting $\set{C}_k$ as defined above,  along with the template hyperparameters of Proposition~\ref{propn:gioa_specialization} and Proposition~\ref{propn:gior_specialization}, respectively.
Once an objective weight vector $\balpha_k^*$ is imputed from one of the inverse models, we solve $\FORT{\balpha^*_k}$ to determine the beamlets $\bw^*_k$ and dose $\bd^*_k$. The dose $\bd^*_k$ is then evaluated using different clinical criteria. Note that we are not attempting to re-construct beamlets or a dose distribution that is similar in $p$-norm to the predictions, but rather learning the objective function weights that the predictions appear to prioritize in order to construct a plan that best reflects clinical preferences in the ground truth $\hat{\bd}_k$. 
Since plan quality is evaluated on dosimetric values in practice, we focus only on the objective space model variants. 


\subsection{The value of ensemble inverse optimization}
\label{sec:imrt_experiments1}

In practice, a suite of quantitative metrics are evaluated to assess whether sufficient dose is delivered to the tumor and the surrounding healthy tissue is sufficiently spared. In line with clinical practice, we use 10 binary criteria for plan evaluation (see the first two columns of Table~\ref{tab:clinCrit}; also \citet{ref:babier_2018b}). 
These criteria cover seven organs-at-risk (OARs) and three planning target volumes (PTVs). OARs are healthy structures whose dose should remain below a specific threshold (e.g., the maximum dose delivered to any voxel in the brainstem should be less than 54 Gy). The PTVs are regions that encompass the tumor sites, and are each assigned a criterion specifying the minimum dose that at least 99\% of its volume should receive. To evaluate our plans on these criteria, we first check whether the corresponding clinical (ground truth) plan satisfied 
given criteria. If the clinical plan satisfied the criteria, we evaluate whether the generated plan also satisfied that criteria.

\begin{table}[t]
        \caption{\vspace{-1em}The percentage of final plans of each KBP population that satisfy the same  clinical criteria as the corresponding clinical plans. OARs are assigned a mean or maximum dose criteria depending on relevance. \vspace{-1em}PTVs are assigned criteria to the 99\%-ile.}
        \centering
        \begin{tabular}{c c  c c c c c } \toprule
                Structure & Criteria (Gy) & $\IORTa{\dataset}$ & \multicolumn{4}{c}{$\IORTr{\dataset}$} \\ 
                & & 8 Pts. & 8 Pts. & 6 Pts. & 4 Pts. & 2 Pts.                  \\ \midrule
                Brainstem & Max $\le$ 54     & 100  & 100  & 100  & 100  & 100  \\ 
                Spinal Cord & Max $\le$ 48   & 100  & 100  & 98.9 & 98.9 & 100  \\
                Right Parotid & Mean $\le$ 26& 58.8 & 88.2 & 88.2 & 82.4 & 94.1 \\
                Left Parotid & Mean $\le$ 26 & 63.6 & 81.8 & 81.8 & 81.8 & 81.8 \\
                Larynx & Mean $\le$ 45       & 59.2 & 95.9 & 95.9 & 93.9 & 95.9 \\
                Mandible & Mean $\le$ 45     & 74.4 & 100  & 100  & 100  & 100  \\ 
                Esophagus & Max $\le$ 73.5   & 51.5 & 100  & 98.5 & 95.5 & 97.0 \\
                PTV70 & 99\%-ile $\ge$ 66.5  & 51.7 & 91.4 & 94.8 & 96.6 & 86.2 \\
                PTV63 & 99\%-ile $\ge$ 59.9  & 50.0 & 98.0 & 98.0 & 98.0 & 98.0 \\
                PTV56 & 99\%-ile $\ge$ 53.2  & 30.4 & 45.7 & 80.4 & 100  & 69.6 \\
                \midrule
                All   &                      & 26.4 & 60.9 & 75.9 & \textbf{83.9} & 70.1 \\
                \bottomrule
        \end{tabular}
        \label{tab:clinCrit}
\end{table}

The columns of Table~\ref{tab:clinCrit} list the proportion of plans generated by $\IORTa{\dataset}$ and $\IORTr{\dataset}$ that satisfied the corresponding clinical criteria. The `All' row reflects the percentage of plans that satisfied all of the criteria that were also met by the corresponding clinical plans and is an aggregate measure of plan quality.
We first use all eight predictions to solve $\IORTa{\dataset}$ (column 3) and $\IORTr{\dataset}$ (column 4).
$\IORTr{\dataset}$ substantially outperforms the $\IORTa{\dataset}$ over every criterion, suggesting that the absolute duality gap model is not well-suited to this application. This result is consistent with results observed for single-point inverse optimization in IMRT~\citep{ref:chan_or14, ref:chan_gof_16, ref:goli_2018} and we conjecture that it is due to the wide range of objective function magnitudes in the forward problem. The absolute duality gap model adjusts each objective value by the same absolute amount, causing relatively large adjustments to objectives with low values and small adjustments to those with high values; thus, it has difficulty balancing different criteria.

Although $\IORTr{\dataset}$ with eight predictions is generally effective at satisfying the OAR criteria, these plans sacrifice the PTV criteria, especially PTV56. We hypothesize that this performance for PTV criteria is due to the large variability in the quality of predictions. For example, the 2-D RGB GAN, 2-D GANCER, and 3-D GANCER models are known to produce plans that emphasize OAR criteria at the expense of the PTV. 
Criteria satisfaction for single-point $\IORTr{\{\bhx\}}$ using each of the individual predictions is shown in Table~\ref{tab:single_points}.
Depending on which prediction is used, the single-point KBP population varies from $10.9\%$ to $95.7\%$ in terms of satisfying the PTV56 criteria. The ability of the single-point models to satisfy all clinical criteria ranges between $44.8\%$ and $80.5\%$, suggesting that some single-point KBP models make poorer trade-offs in criteria satisfaction than others. Regardless of the variability among predictions, the ensemble model outperforms all but the top three single-point models in satisfying all criteria. In cases where the cost of determining model performance is expensive (e.g., having to solve inverse and forward models over multiple predictions and patients), ensemble inverse optimization can reliably provide high-quality plans.

\afterpage{%
\clearpage
\begin{landscape}
\begin{table}[t]
        \centering
        \caption{The percentage of single-point inverse optimization plans of each KBP population that satisfy the same clinical criteria as the clinical plans.} 
        \resizebox{\linewidth}{!}{
        \begin{tabular}{c c c c c c c c c c} \toprule
                Structure & Criteria (Gy) & \multicolumn{8}{c}{$\IORTr{\{ \bhx_q\}}$} \\ 
                &  & 3-D GANCER & 2-D RGB GAN & 2-D GANCER & 2-D RGB GAN-sc. & RF-sc. & RF & 2-D GANCER-sc. & 3-D GANCER-sc. \\ \midrule
                Brainstem & Max $\le$ 54      & 100  & 100  & 100  & 100  & 98.9  & 100  & 100  & 100 \\
                Spinal Cord & Max $\le$ 48    & 100  & 98.9 & 100  & 98.9 & 98.9 & 100  & 98.9 & 98.9 \\
                Right Parotid & Mean $\le$ 26 & 94.1 & 94.1 & 82.4 & 88.2 & 94.1 & 88.2 & 88.2 & 94.1 \\
                Left Parotid & Mean $\le$ 26  & 100  & 90.9 & 81.8 & 63.6 & 72.8 & 63.6 & 81.8 & 81.8 \\
                Larynx & Mean $\le$ 45        & 98.0 & 89.8 & 89.8 & 87.8 & 95.9 & 91.8 & 85.7 & 93.9 \\
                Mandible & Mean $\le$ 45      & 100  & 100  & 100  & 100  & 98.7 & 100	& 100  & 100 \\
                Esophagus & Max $\le$ 73.5    & 100  & 100  & 100  & 98.5 & 100  & 100	& 89.4 & 84.8 \\
                PTV70 & 99\%-ile $\ge$ 66.5   & 81.0 & 36.2 & 81.0 & 69.0 & 63.8 & 91.4 & 98.3 & 100 \\
                PTV63 & 99\%-ile $\ge$ 59.9   & 92.0 & 100  & 100  & 100 & 98.0  & 98.0 & 100  & 100 \\
                PTV56 & 99\%-ile $\ge$ 53.2   & 10.9 & 58.7 & 19.6 & 82.6 & 47.8 & 65.2 & 95.7 & 95.7 \\
                \midrule
                All &                         & 44.8 & 47.1 & 47.1 & 59.8 & 55.2 & 67.8 & 77.0 & \textbf{80.5} \\
                \bottomrule
        \end{tabular}
        }
        \label{tab:single_points}
\end{table}
\end{landscape}
\clearpage
}

\comm{
Using multiple points of varying quality as input to the ensemble model may lead to poor model-data fit (see Example~\ref{ex:gio_rho_scaling}). We experiment with IO models based on subsets of the eight predictions to determine which subset of KBP prediction models best fit $\FORT{\balpha}$. The clinical KBP literature shows that some of the prediction models generally perform better than others: scaled GANCER models typically predict better than RF, which themselves predict better than RGB-GAN and unscaled GANCER~\citep{ref:babier_2018med, ref:mahmood_2018gancer}. 
Using the prior literature and qualitative assessment from a clinical collaborator, we propose an ordering of the models from weak to strong: 3-D GANCER, 2-D RGB GAN, 2-D GANCER, 2-D RGB GAN-sc., RF-sc., RF, 2-D GANCER-sc., 3-D GANCER-sc. Note the general pattern is more important than the exact ordering. That is, we rate the scaled GANCER models as strongest, followed by RF models, followed by RGB GAN and unscaled GANCER.
}
We implement $\IORTr{\dataset}$ with data sets of decreasing size by sequentially removing the two weakest predictors. For example, the 6 Pts. IO model uses the six strongest predictions, while the 4 Pts. model uses only scaled GANCER and RF.
Columns 5--7 of Table~\ref{tab:clinCrit} show the performance of the three subset IO models.  The 6 Pts\@. model markedly improves over the 8 Pts\@. model on PTV criteria, while satisfying almost all OAR criteria, resulting in an additional $15\%$ of the final plans being able to satisfy all criteria. Similarly, the 4 Pts\@. model improves over the 6 Pts\@. model by achieving near perfect PTV criteria satisfaction while mostly preserving OAR performance. In fact, this model now outperforms the best single-point model, 3-D GANCER-sc. (see Table~\ref{tab:single_points}). Interestingly, performance does not improve in the 2 Pts\@. model. This model uses two predictions (2-D GANCER-sc. and 3-D GANCER-sc.) that individually achieve high PTV satisfaction in their single-point models, but fail to do so when combined in an ensemble. We conjecture that the 2 Pts\@. model reaches a local minimum in PTV satisfaction because the forward objectives do not directly target PTV criteria (see~\ref{sec:ec_imrt_gio}).

Overall, these experiments demonstrate that ensemble inverse optimization 
is valuable for turning an ensemble of predictions into a single treatment plan. While an off-the-shelf ensemble model immediately outperforms most single-point constituents, our results show that careful selection of data is required to maximize performance and beat all single-point KBP models.

\begin{table}[t]
    \centering
    \caption{
        \comm{\vspace{-1em}The percentage of plans from different ensemble models that satisfy the same clinical criteria as the corresponding clinical plans. $\IORTr{\dataset}$ refers to the 4 Pts. model from Table~\ref{tab:clinCrit}. \vspace{-1em}We present the best performing setting for each baseline, i.e., the 8 Pts\@. Centroid and the 4 Pts\@. MWA.}
    }
    \begin{tabular}{c c c c c }
            \toprule
            Structure & Criteria (Gy) & $\IORTr{\dataset}$ & Centroid & MWA \\ 
            \midrule
            Brainstem & Max $\le$ 54       & 100  & 100  & 100  \\ 
            Spinal Cord & Max $\le$ 48     & 98.9 & 100  & 100  \\
            Right Parotid & Mean $\le$ 26  & 82.4 & 88.2 & 88.2  \\
            Left Parotid & Mean $\le$ 26   & 81.8 & 81.8 & 63.6  \\
            Larynx & Mean $\le$ 45         & 93.9 & 87.8 & 91.8  \\
            Mandible & Mean $\le$ 45       & 100  & 98.5 & 100  \\ 
            Esophagus & Max $\le$ 73.5     & 95.5 & 100  & 100  \\
            PTV70 & 99\%-ile $\ge$ 66.5    & 96.6 & 96.6 & 93.1  \\
            PTV63 & 99\%-ile $\ge$ 59.9    & 98.0 & 100  & 98.0  \\
            PTV56 & 99\%-ile $\ge$ 53.2    & 100  & 80.4 & 67.4  \\
            \midrule
            All   &                       & \textbf{83.9} & 77.0 & 69.0 \\
            \bottomrule
    \end{tabular}
    \label{tab:clinCrit_ensemble_comparison}
\end{table}

\comm{
\subsection{Comparison with existing ensemble learning techniques}
\label{sec:imrt_comparison_vs_baselines}

We next compare $\IORTr{\dataset}$ with two conventional ensemble learning baselines that do not account for linear programming geometry. The first baseline is an ``ensemble-then-inverse optimization'' approach where for each patient $k$, the centroid $\bar\bd_k$ of the individual predictions is input into a single-point inverse optimization problem $\IORTr{\{\bar{\bd}_k\}}$. 
The second baseline is a Multiplicative Weights Algorithm (MWA), commonly used in ``learning from experts'' settings~\citep{ref:arora_2012multiplicative}. 
Here, we first solve the single-point problem $\IORTr{\{ \hat{\bd}_{k,q}\}}$ with each prediction model for the training set patients.
We treat each prediction model as a different expert and learn a probability distribution over the set of prediction models using the aggregate error as a loss function. Then for each patient in the test set, we use this distribution to randomly sample a prediction model and solve a single-point problem. Baseline implementation details are given in~\ref{sec:ec_baselines}.

We implement the Centroid and the MWA model using all eight predictions per patient (i.e., 8 Pts\@.), as well as the 4 Pts\@. predictions (RF-sc., RF, 2-D GANCER-sc., and 3-D GANCER-sc.). 
Table~\ref{tab:clinCrit_ensemble_comparison} compares our incumbent, the 4 Pts\@. $\IORTr{\dataset}$, with the best-performing Centroid and MWA models. If all dose predictions were feasible with respect to $\FORT{\balpha}$, then by Proposition~\ref{propn:gior_optimal}, our ensemble model and the Centroid model would be equivalent. Each prediction model outputs feasible doses for approximately 85\% of the patients (see Table~\ref{tab:fraction_feas} in the companion). Consequently, $\IORTr{\dataset}$ yields different plans from $\IORTr{\bar{\bd}_{k}}$. Our incumbent outperforms the baseline on the `All' criteria by $6.9\%$. Nonetheless, the Centroid model is similar to $\IORTr{\dataset}$ for each individual criteria. We intuit that if only a small fraction of points in $\dataset$ are infeasible, then centroid inverse optimization is an efficient approximation of ensemble inverse optimization.

The MWA baseline randomly selects a single-point inverse optimization model for each patient according to a learned probability distribution. This approach is a tractable alternative to solving eight inverse optimization problems and selecting the best plan for each patient (see Figure~\ref{fig:kbp_pipeline}a). As shown in Table~\ref{tab:single_points}, some single point models are significantly better than others. Thus, most of the test set patients will receive plans from RF, 2-D GANCER-sc. or 3-D GANCER-sc. However, $\IORTr{\dataset}$ already outperforms each of these single-point models on most of the criteria. Consequently, $\IORTr{\dataset}$ outperforms the MWA baseline on all criteria by $14.9\%$. 
}

\begin{table}[t]
    \centering
    \caption{
        \comm{\vspace{-1em}$\rho$ for the Weak, Medium, and Strong subsets of 2, 4, and 6 Pts. The All criteria percentage satisfaction for each model are in parentheses. The Strong column reflects the predictions used in Table~\ref{tab:clinCrit}.}\vspace{-1em} Highest performing models are bolded. 
    }
    \centering
    \begin{tabular}{c c c c} \toprule
            & \comm{Weak} & \comm{Medium} & \comm{Strong} \\ \midrule
    2 Pts.  & 0.63 (42.5)   & 0.65 (60.9)   & \textbf{0.90 (70.1)} \\
    4 Pts.  & 0.56 (30.1)   & 0.68 (62.1)   & \textbf{0.73 (83.9)} \\
    6 Pts.  & 0.64 (51.7)   & 0.63 (57.5)   & \textbf{0.67 (75.9)} \\
    \bottomrule
    \end{tabular}
    \label{tab:best_avg_worst}
\end{table}

\subsection{Using $\rho$ to validate the best subset of the data}

We previously showed that using a targeted subset of the predictions yielded a better model.
The intuition follows Example~\ref{ex:gio_rho_scaling}, where points that are individually far from each other induce poor fit. While our ranking scheme was domain-specific, here we demonstrate a domain-independent validation of the selection of the data sets in the 6 Pts., 4 Pts., and 2 Pts. models using $\rho$.

\comm{
We consider three variants for each of the 6 Pts., 4 Pts., and 2 Pts. models by selecting subsets of strong, medium, and weak predictions according to our clinical ordering. Strong subsets correspond to the models developed in Section~\ref{sec:imrt_experiments1}, Weak subsets use the lowest ordered predictions and sequentially remove the best, and Medium subsets use the predictions from 2-D RGB GAN to 2-D GANCER-sc. and sequentially remove one strong and weak prediction. 
}
Table~\ref{tab:best_avg_worst} compares $\rho$ across models with varying quality of predictions. 
\comm{
Note that we are not studying the effect of data set size $Q$ (along columns of Table~\ref{tab:best_avg_worst}), but rather the effect of quality (along rows of Table~\ref{tab:best_avg_worst}). For fixed $Q$, the Strong model always yields the highest $\rho$, which suggests that the Strong predictions are the best fit for the clinical forward model. 
}
Furthermore in parentheses in Table~\ref{tab:best_avg_worst}, we show that the clinical criteria satisfaction rates for each of the ensemble models also reflect similar trends as $\rho$. Since $\rho$ is a general metric, we can evaluate the model quality for a given number of points without domain specific knowledge, and come to nearly the same conclusion as via the clinical criteria, which are domain-specific and require additional computation due to re-solving the forward model. 

However, $\rho$ is not a perfect surrogate for criteria satisfaction.
For example, the Weak 6 Pts. model has a slightly higher $\rho$ than the Medium 6 Pts. model. Note that the two data sets share four of six points and the relatively similar $\rho$ reflects a similar criteria satisfaction rate. We also observe that the data set with the best fit from an inverse optimization perspective (Strong 2 Pts.) is not the one resulting in 
the best clinical criteria evaluation (Strong 4 Pts.). 
\comm{
This result is due to the fact $\rho$ is calculated via the average distance of the predictions to the constraints, but the constraints only approximate the criteria (see~\ref{sec:ec_imrt_gio_fo}). 
}
Because the predictions are close to the constraints but not criteria, $\rho$ is overly optimistic for this model. 
\comm{Using diverse predictions of high clinical quality allows us to obtain $\rho$ values that are more representative of the clinical problem.} 

\section{Conclusion}
\label{sec:conclusion}

Inverse optimization is an increasingly popular model-fitting paradigm for estimating the cost vector of an optimization problem given decision data. Motivated by ensemble methods in machine learning, we develop a framework that uses a collection of decisions for a single problem to estimate a cost vector. 
The data is drawn from different decision-makers attempting to solve a single problem or, as in our application, a family of machine learning-generated predictions of an optimal solution. We propose a generalized inverse linear optimization framework that unifies several common variants of inverse optimization from the literature and derive assumption-free exact solution methods for each.
Comparing with the inverse convex optimization literature shows that by focusing on our specialized context, we can leverage the geometry of linear optimization to produce tighter performance bounds and more efficient solution methods.  
To complete our framework, we develop a general goodness of fit metric to measure model-data fit in any inverse linear optimization application. We demonstrate that this metric, by virtue of possessing properties analogous to $R^2$ in linear regression, is easy to calculate and interpret. 


We propose a novel application of ensemble inverse optimization in the automated construction of radiation therapy treatment plans. In contrast to traditional approaches, which generate plans from individual predictions, we use a family of predictions, each with different characteristics and trade-offs, to form treatment plans that better imitate clinically delivered treatments.
Finally, while constructing the best inverse optimization model requires careful clinical expertise, we show how our goodness-of-fit metric provides domain-independent validation of our model engineering. Beyond the specific context and application presented in this paper, we believe there will be new applications of predict-and-ensemble inverse optimize frameworks that can build on our foundation. 


\bibliographystyle{informs2014}
\bibliography{references}

\ECSwitch

\ECHead{Electronic Companion}

\section{Proofs of Statements}


\begin{proof}{Proof of Proposition~\ref{propn:gioa_specialization}.}
        For any $\bc$, setting each $\bepsilon_q = \epsilon_q \cancelvec{\bc}$ implies $\norm{\bepsilon_q}_\infty = |\epsilon_q| \norm{\cancelvec{\bc}}_\infty = |\epsilon_q|$. Thus,~\eqref{eq:gio1} becomes~\eqref{eq:gioa1}. Similarly,~\eqref{eq:gio3} becomes~\eqref{eq:gioa3}, since $\bc^\tpose \bepsilon_q = \epsilon_q \bc^\tpose \cancelvec{\bc} = \epsilon_q$. Then, any feasible solution to $\GMIO{\dataset}$ with the suggested hyperparameters yields a feasible solution to $\GIOa{\dataset}$ and vice versa, with the same objective value. \qedwhite
        %
\end{proof}

\begin{proof}{Proof of Theorem~\ref{thm:gio_reform}.}
        Let $j^* \in \underset{j \in \set{J}}{\argmax}\left\{ | c^*_j | \right\}$, implying $| c^*_{j^*} | = 1$. Then, $\left( \bc^*, \by^*, \epsilon_1^*,\dots,\epsilon_Q^* \right)$ is feasible to $\GIOap{\dataset;j^*}$. Conversely, for any $j \in \set{J}$, every feasible solution to $\GIOap{\dataset;j}$ is 
        feasible to $\GIOa{\dataset}$, so all optimal solutions to each $\GIOap{\dataset;j}$ lie in the feasible set of $\GIOa{\dataset}$. \qedwhite
\end{proof}

\begin{proof}{Proof of Proposition~\ref{propn:gioa_optimal}.}
        If all observations are feasible, then by weak duality $\epsilon_q \geq 0$ $\forall q \in \set{Q}$, and we can simplify the objective function
        $      \sum_{q=1}^Q |\epsilon_q| = \sum_{q=1}^Q \epsilon_q
        = \sum_{q=1}^Q \left(\bc^\tpose \bhx_q - \bb^\tpose \by \right)
        = \left( \bc^\tpose \bbx - \bb^\tpose \by \right)Q$,
        where the last equality follows by the definition of the centroid (i.e., $\bbx = \sum_{q=1}^Q \bhx_q / Q$). We similarly compress constraint~\eqref{eq:gioa3} to a single constraint for $\bbx$, resulting in $\GSIOa{\bbx}$. \qedwhite
\end{proof}

\begin{proof}{Proof of Proposition~\ref{propn:gioa_infeas_sol}.}~%
        \begin{enumerate}
                \item Assume without loss of generality that there exist $i,j \in \set{I}$ such that $\ba_i^\tpose \bhx > b_i$ and $\ba_j^\tpose \bhx < b_j$, respectively. The corresponding $\tilde{y}$ defined in~\eqref{eq:gioa_sp_infeas_sol} satisfies the strong duality constraint~\eqref{eq:gioa3} with $\epsilon = 0$. Furthermore, $(\btc, \bty)$ satisfy the duality feasibility constraints~\eqref{eq:gioa2} by construction. We normalize the solution to satisfy constraint~\eqref{eq:gioa5}. The normalized solution still satisfies all other constraints. This solution is  feasible for $\GIOa{\{\bhx\}}$ with zero cost and is thus optimal.

                \item Here, the duality gap is non-positive (i.e., $\epsilon \leq 0$). We rewrite the single-point version of~\eqref{eq:gioa} with $\delta = -\epsilon$, shown in model~\eqref{eq:gioa_reform_after_change} below. Now consider the forward problem $\underset{\bx}{\min} \{-\bc^\tpose \bx \;|\; \bA \bx \leq \bb\}$ with the observed solution $\bhx$ and the corresponding inverse optimization model~\eqref{eq:gioa_reverse_gio}.
        \end{enumerate}
                %
        %
        %

        \noindent\begin{minipage}{.5\linewidth}
                \begin{align}
                        \begin{split}
                                \minimize{\bc,\by,\delta} \quad  & \delta \\
                                \subjectto \quad    & \bA^\tpose \by = \bc, \; \by \geq \bzero \\
                                & \bc^\tpose \bhx = \bb^\tpose \by - \delta \\
                                & \nnorm{\bc} = 1. \\
                        \end{split} \label{eq:gioa_reform_after_change}
                \end{align}
        \end{minipage}%
        \begin{minipage}{.5\linewidth}
                \begin{align}
                        \begin{split}
                                \minimize{\bc, \by, \gamma} \quad    & | \gamma | \\
                                \subjectto \quad    & \bA^\tpose \by = \bc, \; \by \geq \bzero \\
                                & -\bc^\tpose \bhx = -\bb^\tpose \by + \gamma  \\
                                & \nnorm{\bc} = 1. \\
                        \end{split} \label{eq:gioa_reverse_gio}
                \end{align}
        \end{minipage}

        By assumption, $\bhx$ is feasible for the above-defined forward problem 
        and therefore, $\gamma \geq 0$ in~\eqref{eq:gioa_reverse_gio}. Consequently, formulation~\eqref{eq:gioa_reform_after_change} is equivalent to~\eqref{eq:gioa_reverse_gio} after removing the absolute value in the objective and rearranging the duality gap constraint. We can solve formulation~\eqref{eq:gioa_reverse_gio} using Theorem~\ref{thm:gio_sp_sol}, arriving at an optimal solution for the original inverse optimization problem. \qedwhite
\end{proof}

\begin{proof}{Proof of Corollary~\ref{cor:gioa_infeas_sol_mp}.}
        Since all observations are infeasible for the initial forward problem, the duality gap terms are all non-positive (i.e., $\epsilon_q \leq 0$ for all $q \in \set{Q}$). As such, we use the same argument as used in Prop.~\ref{propn:gioa_infeas_sol} Part 2 to show that the formulation of $\GIOa{\dataset}$ is equivalent to the formulation of an absolute duality gap inverse optimization problem over the alternative forward problem $\underset{\bx}{\min} \{-\bc^\tpose \bx \;|\; \bA \bx \leq \bb\}$. 
        As $\dataset \subset \left\{ \bx \mid \bA \bx \leq \bb \right\}$, Proposition~\ref{propn:gioa_optimal} reduces the problem to $\GIOa{\{\bbx\}}$. \qedwhite
\end{proof}

\begin{proof}{Proof of Proposition~\ref{propn:gior_specialization}.}
        For any $\bc$, setting $\bepsilon_q = \bb^\tpose \by \left( \epsilon_q - 1 \right) \cancelvec{\bc}$ forces $\norm{\bepsilon_q}_\infty / | \bb^\tpose \by | = | \epsilon_q - 1 |$, giving us the objective~\eqref{eq:gior1}. The same substitution into~\eqref{eq:gio3} gives the strong duality constraint~\eqref{eq:gior3}. Thus, every feasible solution of $\GIOr{\dataset}$ has a corresponding feasible solution in $\GMIO{\dataset}$ (after setting the hyperparameters), and vice versa, with the same objective value. 
        \qedwhite
\end{proof}

\begin{remark}
Proposition~\ref{propn:gior_specialization} addresses the case where $\bb^\tpose \by^* \neq 0$ only. However, if $\bb^\tpose \by^* = 0$, $\GIOr{\dataset}$ and $\GMIO{\dataset}$ are still equivalent in that they both yield an optimal value of $0$. To see this, suppose that an optimal solution to $\GIOr{\dataset}$ satisfies $\bb^\tpose \by^* = 0$. Then for all $q\in\set{Q}$, $\bc^{*\tpose} \bhx_q = 0$ and since $\epsilon_q$ becomes a free variable, we set it to $1$ and obtain an optimal value of $0$. On the other hand, we can use the same $\left( \bc^*, \by^*, \bzero,\dots,\bzero \right)$ as a feasible solution to $\GMIO{\dataset}$ and observe that setting $\bepsilon_q = \bzero$ for all $q \in \set{Q}$ satisfies the strong duality constraint, giving an optimal value of $0$. 
\end{remark}

\begin{proof}{Proof of Proposition~\ref{propn:gior_sol}.}
        Let $(\bhc, \bhy)$ be an optimal solution to $\GIOr{\dataset}$ and let
        \begin{align}
                \label{eq:gior_K}
                K = \begin{cases}
                        1/|\bb^\tpose \bhy|     & \text{~if $\bb^\tpose \bhy \neq 0$} \\
                        1/\bhy^{\tpose} \bone   & \text{~otherwise.}
                \end{cases}
        \end{align}
        We omit the variables $(\epsilon_1,\dots,\epsilon_Q)$ when writing optimal solutions for conciseness. First, we show that $(\bhc, \bhy)$ maps to a corresponding feasible solution for one of $\GIOrp{\dataset;K}$, $\GIOrn{\dataset;K}$, or $\GIOrz{\dataset;K}$ with the same objective value. Conversely, every feasible solution to formulations~\eqref{eq:gior_reform_positive}--\eqref{eq:gior_reform_zero} has a corresponding feasible solution in $\GIOr{\dataset}$ with the same objective value. 

        First, suppose $\bb^\tpose \bhy > 0$ and consider $(\btc, \bty) = \left( \bhc/\bb^\tpose \bhy, \bhy/\bb^\tpose \bhy \right)$. This solution is feasible to $\GIOrp{\dataset;K}$ as $\bb^\tpose \bty = 1$ and $\nnorm{\btc} = K$. Furthermore, by substituting $\btc = \bhc/\bb^\tpose \bhy$, we see that the objective value of this solution for $\GIOrp{\dataset;K}$ is equal to the optimal value for $\GIOr{\dataset}$:
        %
        $\sum_{q=1}^Q \left| \btc^\tpose \bhx_q - 1 \right| = \sum_{q=1}^Q \left| \left(\bhc^{\tpose}\bhx_q \right) / \left( \bb^\tpose \bhy \right) - 1 \right|$.
        %
        Similarly, when $\bb^\tpose \bhy < 0$, we construct $(\btc, \bty) = \left( \bhc/|\bb^\tpose \bhy|, \bhy/|\bb^\tpose \bhy| \right)$, which is feasible to $\GIOrn{\dataset;K}$ and incurs the same objective value as the optimal value of $\GIOr{\dataset}$. Finally, if $\bb^\tpose \bhy = 0$, then the optimal value of $\GIOr{\dataset}$ is $0$. Let $(\btc, \bty) = \left( \bhc/\bhy^{\tpose} \bone, \bhy/\bhy^{\tpose} \bone\right)$. It is straightforward to show that this solution is feasible for $\GIOrz{\dataset;K}$. Thus, an optimal solution to $\GIOr{\dataset}$ can be scaled to construct a solution that is feasible for exactly one of the formulations~\eqref{eq:gior_reform_positive}--\eqref{eq:gior_reform_zero}.

        The converse is proven by showing that every feasible solution of~\eqref{eq:gior_reform_positive}--\eqref{eq:gior_reform_zero} can be scaled to a feasible solution of $\GIOr{\dataset}$. Let $(\btc, \bty)$ be a feasible solution to one of~\eqref{eq:gior_reform_positive}--\eqref{eq:gior_reform_zero}, and let $(\bhc, \bhy) = \left( \btc/\nnorm{\btc}, \bty/\nnorm{\btc} \right)$. This solution is feasible for $\GIOr{\dataset}$ with the same objective function value.

        In terms of objective value, all feasible solutions of $\GIOrp{\dataset;K}$, $\GIOrn{\dataset;K}$, and $\GIOrz{\dataset;K}$ have a one-to-one correspondence with feasible solutions of $\GIOr{\dataset}$ and the best optimal solution to formulations~\eqref{eq:gior_reform_positive}--\eqref{eq:gior_reform_zero} can be scaled to an optimal solution for $\GIOr{\dataset}$. \qedwhite
\end{proof}
%

\begin{proof}{Proof of Proposition~\ref{propn:gior_optimal}.}
        When all of the observed points are feasible, $\bc^\tpose\bhx_q - \bb^\tpose\by \ge 0$, $\forall q \in \set{Q}$. Thus, objective~\eqref{eq:gior1} becomes
        $     \sum_{q=1}^Q | \epsilon_q - 1 | = \sum_{q=1}^Q \frac{\bc^\tpose \bhx_q - \bb^\tpose \by}{|\bb^\tpose \by|}
        = Q \left( \frac{\bc^\tpose \bbx - \bb^\tpose \by}{|\bb^\tpose \by|} \right). $
        Noting that $\bbx$ must also be feasible, the last term equals the objective for $\GSIOr{\bbx}$. \qedwhite
\end{proof}

\begin{proof}{Proof of Lemma~\ref{lem:gio_opt}.}
        Without loss of generality, assume that $\nnorm{\ba_i} = 1$ for all $i \in \set{I}$.
        Solution~\eqref{eq:gio_has_optimal} is feasible to $\GIOd{\dataset}$ for all $i \in \set{I}$. We show that for any feasible solution that is not of the form~\eqref{eq:gio_has_optimal}, there exists a feasible solution of that form whose objective value is at least as good.

        Consider a feasible solution $(\btc, \bty, \btepsilon_1, \dots, \btepsilon_Q)$ to $\GIOd{\dataset}$, where $\bty \neq \BFe_i$ for any $i \in \set{I}$. Without loss of generality, assume $\tilde{y}_1, \dots, \tilde{y}_k > 0$ for some $1 < k \leq m$ and let $\set{K} = \left\{ 1, \dots, k  \right\}$ denote the corresponding index set. Let $\btx_q = \bhx_q - \btepsilon_q$ denote the perturbed decision for all $q \in \set{Q}$. 
        The primal feasibility constraint~\eqref{eq:giod4} implies that $\bA \btx_q \geq \bb$ for all $q \in \set{Q}$. The strong duality constraint~\eqref{eq:giod3} implies that for all $q \in \set{Q}$,
        %
        $ 0 = \bc^\tpose \btx_q - \bb^\tpose \bty = \sum_{i=1}^k \tilde{y}_i \left( \ba_i^\tpose \btx_q - b_i \right)$,
        which follows from substituting $\btc = \sum_{i=1}^k \tilde{y}_i \ba_i$.
        %
        Using the non-negativity of $\bty$ and primal feasibility (i.e., $\ba_i^\tpose \btx_q \geq b_i$ for all $i \in \set{I}$), we see that $\btx_q$ for all $q \in \set{Q}$ are feasible solutions to the feasible projection problem~\eqref{eq:def_fproji} for each $i \in \set{K}$.
%

        Let $\left( \bhc, \bhy, \bhepsilon_1,\dots, \bhepsilon_Q \right) = \left( \ba_{i^*}, \BFe_{i^*}, \bhx_1 - \fprojistar{\bhx_1}, \dots, \bhx_Q - \fprojistar{\bhx_Q} \right)$ for an arbitrary index $i^* \in \set{K}$.
        For all $q \in \set{Q}$, $\fprojistar{\bhx_q}$ is, by definition, an optimal solution to~\eqref{eq:def_fproji}.
        Therefore, we have
        %
        $\sum_{q=1}^Q \norm{\bhepsilon_q}_p = \sum_{q=1}^Q \norm{\bhx_q - \fprojistar{\bhx_q}}_p 
        \leq \sum_{q=1}^Q \norm{\btepsilon_q}_p, \label{eq:proof_gio_opt2}$
        %
        with the inequality following from the optimality of~\eqref{eq:def_fproji}. Thus, given any feasible solution to $\GIOd{\dataset}$ not of the form defined in~\eqref{eq:gio_has_optimal}, we can construct a feasible solution of the form~\eqref{eq:gio_has_optimal} with the objective value at least as good as the original.\qedwhite
\end{proof}

\begin{proof}{Proof of Theorem~\ref{thm:gio_opt}.}
        For each $i$, the inner optimization problem produces solutions with the structure in~\eqref{eq:gio_has_optimal}. Thus, the inner optimization problems, along with the corresponding $\left(\bc, \by\right)$ enumerate all possible solutions to $\GIOd{\dataset}$ with the structure in~\eqref{eq:gio_has_optimal}. By Lemma~\ref{lem:gio_opt}, we select the one yielding the lowest objective value.\qedwhite
\end{proof}

\begin{proof}{Proof of Theorem~\ref{propn:dominance_relationships}.}~
    First note that due to the dominance between $p$-norms, (i.e., $\norm{\bepsilon}_p \geq \norm{\bepsilon}_\infty$) we have $z^*_p \geq z^*_\infty$, since the choice of $p$ only affects the objective and the two problems share the same feasible set. We then lower bound the optimal value of $\GIOp{\dataset}{\infty}$ using Theorem~\ref{thm:gio_opt}:
    \begin{align}
            \left.\begin{array}{rl}
                            \displaystyle\min_{i \in \set{I}} \; \min_{\bepsilon_{1,i},\dots,\bepsilon_{Q,i}} \quad & \displaystyle\sum_{q=1}^Q \norm{\bepsilon_{q,i}}_\infty \\
                            \st \quad\quad & \displaystyle\bA \left( \bhx_q - \bepsilon_{q,i} \right) \geq \bb_i, \forall q \in \mathcal{Q} \\
                            & \displaystyle\ba_i^\tpose \left( \bhx_q - \bepsilon_{q,i} \right) = b_i, \forall q \in \mathcal{Q}
            \end{array}\right\}
            &= \min_{i\in\set{I}} \left\{ \sum_{q=1}^Q \norm{ \bhx_q - \fproji{\bhx_q}}_\infty \right\} \\
            &\geq \min_{i \in \set{I}} \left\{ \sum_{q=1}^Q \norm{ \bhx_q - \proji{\bhx_q}}_\infty \right\} \label{eq:gio_sol1} \\
            &= \min_{i \in \set{I}} \left\{ \sum_{q=1}^Q \frac{\left|\ba_i^\tpose \bhx_q-b_i\right|}{{\norm{\ba_i}}_1} \right\} \label{eq:gio_sol2}\\
            &= \min_{i \in \set{I}} \left\{ Q \left(\frac{\ba_i^\tpose \bbx-b_i}{{\norm{\ba_i}}_1}\right) \right\} \label{eq:gio_sol3}.
    \end{align}
    The inequality in~\eqref{eq:gio_sol1} comes from the fact that the projection problem~\eqref{eq:def_proji} is a relaxation of the feasible projection problem~\eqref{eq:def_fproji}, by removing the feasibility constraint. The equality of~\eqref{eq:gio_sol2} comes from~\citet{ref:mangasarian_99} (e.g., see Theorem~\ref{thm:gio_sp_sol}), which provides the analytic optimal value of the projection problem. Because $\bhx_q \in \feas$ for all $\bhx_q \in \dataset$, we bypass the absolute values to average. Note that~\eqref{eq:gio_sol3} is equal to the optimal value of $\GSIO{\bbx}$.

    Now consider $\GIOa{\dataset}$. Because, $\dataset \subset \set{P}$, Proposition~\ref{propn:gioa_optimal} yields $z^*_\mathrm{A} = z^*\big(\GSIO{\bbx}\big)$, i.e., the optimal solution to $\GSIO{\bbx}$ where $\bbx = \sum_{q=1}^Q \bhx_q / Q$ is the centroid of $\dataset$. In conjunction with~\eqref{eq:gio_sol3}, we conclude that $z^*_p\geq z^*_\infty \geq z^*_\mathrm{A}$.
\qedwhite\end{proof}

\comm{
\begin{proof}{Proof of Corollary~\ref{cor:dominance_adg_rdg}.}~We remark that Corollary~\ref{cor:dominance_adg_rdg} is in fact a special case of a more general statement regarding error measures in the absolute versus relative space. Below, we prove a more general statement and specialize the result to the case of inverse optimization.

Let $f(\bx)$ and $g(\bx)$ be two functions and $f(\bx) \neq \bzero$ for all $\bx$. Consider two optimization problems:

\noindent\begin{minipage}{.5\linewidth}
    \begin{align}
            \begin{split}
                \minimize{\bx} \quad  & \sum_{q=1}^Q \left| g_q(\bx) - f(\bx) \right| \\
                \subjectto \quad    & \bx \in \set{X} \\
            \end{split} \label{eq:gioa_remove_epsilon}
    \end{align}
\end{minipage}%
\begin{minipage}{.5\linewidth}
    \begin{align}
            \begin{split}
                \minimize{\bx} \quad  & \sum_{q=1}^Q \left| \frac{g_q(\bx) - f(\bx)}{f(\bx)} \right| \\
                \subjectto \quad    & \bx \in \set{X} \\
            \end{split} \label{eq:gior_remove_epsilon}
    \end{align}
\end{minipage}
Let $\bx^*_\mathrm{A}$ and $z^*_\mathrm{A}$ be an optimal solution and value, respectively for~\eqref{eq:gioa_remove_epsilon}. Similarly, let $\bx^*_\mathrm{R}$ and $z^*_\mathrm{R}$ be an optimal solution and value, respectively for~\eqref{eq:gior_remove_epsilon}. We will prove that $|f(\bx^*_\mathrm{R})| z^*_\mathrm{R} \geq z^*_\mathrm{A} \geq z^*_\mathrm{R} |f(\bx^*_\mathrm{A})|$.

First note that $\bx^*_\mathrm{A}$ is feasible for~\eqref{eq:gior_remove_epsilon} and $\bx^*_\mathrm{R}$ is feasible for~\eqref{eq:gioa_remove_epsilon}. Then, 
\begin{align*}
    z^*_\mathrm{A} 
        = \sum_{q=1}^Q \left| g_q(\bx^*_\mathrm{A}) - f(\bx^*_\mathrm{A}) \right| 
        \leq \sum_{q=1}^Q \left| g_q(\bx^*_\mathrm{R}) - f(\bx^*_\mathrm{R}) \right| 
        = \sum_{q=1}^Q \left| \frac{g_q(\bx^*_\mathrm{R}) - f(\bx^*_\mathrm{R})}{f(\bx^*_\mathrm{R})} \right| |f(\bx^*_\mathrm{R})| = z^*_\mathrm{R} |f(\bx_\mathrm{R})|.
\end{align*}
The inequality comes from the feasibility of $\bx^*_\mathrm{R}$ for~\eqref{eq:gioa_remove_epsilon} and the second equality comes from multiplying by $|f(\bx^*_\mathrm{R})|/|f(\bx^*_\mathrm{R})|$. This proves the left inequality. 

We next show
\begin{align*}
    z^*_\mathrm{R} = \sum_{q=1}^Q \left| \frac{g_q(\bx^*_\mathrm{R}) - f(\bx^*_\mathrm{R})}{f(\bx^*_\mathrm{R})} \right|
        \leq \sum_{q=1}^Q \left| \frac{g_q(\bx^*_\mathrm{A}) - f(\bx^*_\mathrm{A})}{f(\bx^*_\mathrm{A})} \right|
        = \sum_{q=1}^Q \left| g_q(\bx^*_\mathrm{A}) - f(\bx^*_\mathrm{A})  \right| \frac{1}{|{f(\bx^*_\mathrm{A})}|}
        = z^*_\mathrm{A} \frac{1}{|{f(\bx^*_\mathrm{A})}|}.
\end{align*}
The inequality comes from the feasibility of $\bx^*_\mathrm{A}$ for~\eqref{eq:gior_remove_epsilon}. This proves the right inequality.

Finally, we observe that letting $\bx = (\bc, \by)$, $\set{X} = \{ (\bc, \by) \;|\; \bA^\tpose \by = \bc, \by \geq \bzero, \nnorm{\bc} = 1$, $g_q(\bx) = \bc^\tpose \bhx_q$, and $f(\bx) = \bb^\tpose \by$ converts~\eqref{eq:gioa_remove_epsilon} into $\GIOa{\dataset}$ and~\eqref{eq:gior_remove_epsilon} into $\GIOr{\dataset}$. Finally note that for any feasible pair $(\bc, \by)$, $\bb^\tpose \by$ is equal to the optimal value of the forward problem $\FO{\bc}$. Substituting the terms for the absolute and relative duality gap problems respectively completes the inequality. \qedwhite
\end{proof}
}

\begin{proof}{Proof of Theorem~\ref{thm:rho_properties}.}~%
        \begin{enumerate}
                \item Given $\dataset$, $\bA$, and $\bb$, the denominator term in $\rho$ is fixed. An optimal solution to $\GMIO{\dataset}$ minimizes the numerator of $1-\rho$, thus maximizing $\rho$.

                \item We prove $1-\rho \in [0,1]$. It is easy to see that $1-\rho \geq 0$, because it is the ratio of sums of norms, which are nonnegative. To show $1-\rho \leq 1$, note that $\sum_{q=1}^Q \norm{\bepsilon^*_q} \leq \sum_{q=1}^Q \norm{\bepsilon_{q,i}}$ for all $i$, as setting $\bc = \ba_i/\nnorm{\ba_i}$ will yield a feasible but not necessarily optimal solution to $\GMIO{\dataset}$.

        \item An optimal solution to $\GIOK{\dataset}{k}$ is feasible for $\GIOK{\dataset}{k+1}$, since the latter problem is a relaxation of the former. Invoking the first statement in this theorem, $\rho^{(k)} \leq \rho^{(k+1)}$.          \qedwhite\end{enumerate}

\end{proof}

\section{A general solution method for $\GIOr{\dataset}$} 
\label{sec:ec_gior_K}

Although Proposition~\ref{propn:gior_sol} reformulates $\GIOr{\dataset}$ into three sub-problems, the norm constraint $\nnorm{\cdot} \geq K$ in the sub-problems adds two challenges: first, the constraint itself is non-convex, and second, an appropriate value for $K$ must be chosen in order for Proposition~\ref{propn:gior_sol} to hold. As the non-convex constraint can be handled by polyhedral decomposition, we first discuss how to choose a valid $K$. We then consider a relaxed reformulation of $\GIOr{\dataset}$ that often works well in practice. Finally, we summarize all of these results into a general solution algorithm for inverse optimization minimizing the relative duality gap. These steps are summarized in Algorithm~\ref{alg:gior}.

The proof of Proposition~\ref{propn:gior_sol} shows that for any $K > 0$, every feasible solution of $\GIOrp{\dataset;K}$, $\GIOrn{\dataset;K}$, and $\GIOrz{\dataset;K}$ can be mapped to a feasible solution of $\GIOr{\dataset}$. The normalization constraint $\nnorm{\bc} \geq K$ implies that the feasible region for each sub-problem grows as $K$ decreases. The proof then shows that for some sufficiently small $K > 0$, an optimal solution to $\GIOr{\dataset}$ can be mapped to a feasible (and therefore, also optimal) solution of one of~\eqref{eq:gior_reform_positive}--\eqref{eq:gior_reform_zero}.

To determine a sufficiently small $K$, note that the mapping of a solution of $\GIOr{\dataset}$ to solutions of one of~\eqref{eq:gior_reform_positive}--\eqref{eq:gior_reform_zero} involves scaling the solution by $\bb^\tpose \by$, $-\bb^\tpose \by$, or $\by^\tpose \bone$, respectively. Bounding these terms allows us to determine a sufficiently small $K$. Formally, consider the following
problem:
\begin{align}
        \begin{split}
                \maximize{\by} \quad    & \max \left\{ | \bb^\tpose \by |, \by^\tpose \bone \right\} \\
                \subjectto \quad    & \nnorm{\bA^\tpose \by} = 1, \; \by \geq \bzero. \\
        \end{split} \label{eq:gior_K_formulation}
\end{align}
We refer to formulation~\eqref{eq:gior_K_formulation} as the auxiliary problem for $\GIOr{\dataset}$. The auxiliary problem can be written as three optimization problems, each with the same constraints as~\eqref{eq:gior_K_formulation} but a different objective: $\bb^\tpose \by$, $-\bb^\tpose \by$, and $\by^\tpose \bone$. Since the auxiliary problem has a normalization constraint similar to the one in $\GIOa{\dataset}$, we can use the same methods to solve it. Let $K^*$ be defined as the reciprocal of the optimal value of the auxiliary problem. Note that $K^*$ is well-defined. That is, the auxiliary problem always has a non-zero solution, because any feasible $\by$ to~\eqref{eq:gior_K_formulation} must have $\by \geq \bzero$ and at least one non-zero $y_i > 0$, meaning $\by^\tpose \bone > 0$ must always hold. We use $K^*$ to reformulate $\GIOr{\dataset}$ to $\GIOrp{\dataset;K^*}$, $\GIOrn{\dataset;K^*}$, and $\GIOrz{\dataset;K^*}$.

\begin{theorem}
        \label{thm:gior_K_sol}
        Let $z^+$ be the optimal value of $\GIOrp{\dataset;K^*}$ if it is feasible, otherwise $z^+ = \infty$. Let $z^-$ and $z^0$ be defined similarly for $\GIOrn{\dataset;K^*}$ and $\GIOrz{\dataset;K^*}$, respectively. Let $z^* = \min\left\{ z^+, z^-, z^0 \right\}$ and let $\left(\bc^*, \by^*,\epsilon_1^*,\dots,\epsilon_Q^*\right)$ be the corresponding optimal solution. Then, $\left(\bc^*/\nnorm{\bc^*}, \by^*/\nnorm{\bc^*}, \epsilon_1^*, \dots, \epsilon_Q^*\right)$ is optimal to $\GIOr{\dataset}$.
\end{theorem}
\begin{proof}{Proof of Theorem~\ref{thm:gior_K_sol}.}
        Let $\left( \bhc, \bhy \right)$ be optimal to $\GIOr{\dataset}$ and $K$ be defined as in~\eqref{eq:gior_K}. Since $\bhy$ is feasible for the auxiliary problem~\eqref{eq:gior_K_formulation}, $1/K^* \geq \max\left\{ |\bb^\tpose \bhy|, \bhy^\tpose \bone \right\}$, implying $K^* \leq K$.

        The proof of Proposition~\ref{propn:gior_sol} showed that scaling $\left( \bhc, \bhy \right)$ appropriately yielded a corresponding feasible solution to one of $\GIOrp{\dataset;K}$, $\GIOrn{\dataset;K}$, or $\GIOrz{\dataset;K}$. Because $K^* \leq K$, the scaled solution must also be feasible for the respective $\GIOrp{\dataset;K^*}$, $\GIOrn{\dataset;K^*}$, or $\GIOrz{\dataset;K^*}$. Moreover, every solution of $\GIOrp{\dataset;K^*}$, $\GIOrn{\dataset;K^*}$, or $\GIOrz{\dataset;K^*}$ can be scaled to a feasible solution of $\GIOr{\dataset}$, completing the proof. \qedwhite
\end{proof}
%


%
\begin{algorithm}[t]
        \caption{General solution method for $\GIOr{\dataset}$}
        \label{alg:gior}
        \begin{algorithmic}[1]
                \renewcommand{\algorithmicrequire}{\textbf{Input:}}
                \renewcommand{\algorithmicensure}{\textbf{Output:}}
                \REQUIRE Data set $\dataset$
                \ENSURE  Imputed model parameters $\left( \bc^*, \by^*, \epsilon_1^*, \dots, \epsilon_Q^* \right)$
                \STATE Let $z^+_\LP \gets \GIOrprel{\dataset}$, $z^-_\LP \gets \GIOrnrel{\dataset}$, $z^0_\LP \gets \GIOrzrel{\dataset}$ be the optimal values. 
                \STATE Let $z^*_\LP \leftarrow \min\left\{ z^+_\LP, z^-_\LP, z^0_\LP  \right\}$ and $\left( \bc^*, \by^*, \epsilon_1^*,\dots, \epsilon_Q^* \right)$ be the corresponding optimal solution.
                \IF{$\bc^* \neq \bzero$}
                \RETURN $\left( \bc^*, \by^*, \epsilon_1^*, \dots, \epsilon_Q^* \right)$
                \ELSE
                \STATE Solve the auxiliary problem~\eqref{eq:gior_K_formulation}. Let $K^*$ be the reciprocal of the optimal value.
                \STATE Let $z^+ \gets \GIOrp{\dataset;K^*}$, $z^- \gets \GIOrn{\dataset;K^*}$, $z^0 \gets \GIOrz{\dataset;K^*}$ be the optimal values. 
                \STATE Let $z^* \leftarrow \min\left\{ z^+, z^-, z^0 \right\}$ and $\left( \bc^*, \by^*, \epsilon_1^*, \dots, \epsilon_Q^* \right)$ be the corresponding optimal solution.
                \RETURN $\left( \bc^*, \by^*, \epsilon_1^*, \dots, \epsilon_Q^* \right)$
                \ENDIF
        \end{algorithmic}
\end{algorithm}

In the most general case,
solving $\GIOr{\dataset}$ is more computationally intensive than solving $\GIOa{\dataset}$. We must first identify $K^*$, which we can use to reformulate $\GIOr{\dataset}$ into three norm-constrained optimization problems. Subsequently, given an appropriate choice of $\nnorm{\cdot}$, each problem is decomposed into a series of 
LPs.
For instance,
doing so leads to $2n$ LPs if $\nnorm{\cdot} = \norm{\cdot}_\infty$ and $2^n$ LPs if $\nnorm{\cdot} = \norm{\cdot}_1$.
These steps coupled with the auxiliary problem~\eqref{eq:gior_K_formulation} used to determine $K^*$ require the solution of $12n$ 
LPs
when $\nnorm{\cdot} = \norm{\cdot}_\infty$, or $6(2^n)$ when $\nnorm{\cdot} = \norm{\cdot}_1$. In some cases, however, it may be possible to find an optimal solution to $\GIOr{\dataset}$ by solving exactly three LPs. 

\begin{corollary}
        \label{cor:gior_sol_nonzero}
        Let $\GIOrprel{\dataset}$, $\GIOrnrel{\dataset}$, and $\GIOrzrel{\dataset}$ be the LP relaxations of $\GIOrp{\dataset;K}$, $\GIOrn{\dataset;K}$, and $\GIOrz{\dataset;K}$, respectively, obtained by removing the normalization constraint $\nnorm{\bc} \geq K$. Let $z^+_\LP$ be the optimal value of $\GIOrprel{\dataset}$ if it is feasible, otherwise $z^+_\LP = \infty$. Let $z^-_\LP$ and $z^0_\LP$ be defined similarly for $\GIOrnrel{\dataset}$ and $\GIOrzrel{\dataset}$, respectively. Let $z^*_\LP = \min\left\{ z^+_\LP, z^-_\LP, z^0_\LP \right\}$ and let $\left(\bc^*, \by^*, \epsilon_1^*, \dots, \epsilon_Q^*\right)$ be an optimal solution of the corresponding problem. If $\bc^* \neq \bzero$, then $z^*_\LP$ is equal to the optimal value of $\GIOr{\dataset}$ and $\left(\bc^*/\nnorm{\bc^*}, \by^*/\nnorm{\bc^*}, \epsilon_1^*,\dots,\epsilon_Q^*\right)$ is an optimal solution to $\GIOr{\dataset}$.
\end{corollary}

\begin{proof}{Proof of Corollary~\ref{cor:gior_sol_nonzero}.}
        Let $(\bhc, \bhy, \hat{\epsilon}_1,\dots,\hat{\epsilon}_Q)$ be an optimal solution to $\GIOr{\dataset}$. From Proposition~\ref{propn:gior_sol}, this solution can be rescaled to construct a feasible solution for one of $\GIOrprel{\dataset}$, $\GIOrnrel{\dataset}$, and $\GIOrzrel{\dataset}$ with the same objective value. Conversely, for each of the relaxed problems, let $(\btc, \bty, \tilde{\epsilon}_1,\dots,\tilde{\epsilon}_Q)$ be a feasible solution. Assuming that $\btc \neq \bzero$, this solution can be rescaled to construct $(\bhc, \bhy, \hat{\epsilon}_1,\dots,\hat{\epsilon}_Q) = \left( \btc/\nnorm{\btc}, \bty/\nnorm{\btc}, \tilde{\epsilon}_1,\dots,\tilde{\epsilon}_Q \right)$, which is a feasible solution to $\GIOr{\dataset}$ with the same objective value. Thus, if the minimum of $\GIOrprel{\dataset}$, $\GIOrnrel{\dataset}$, and $\GIOrzrel{\dataset}$ yields an optimal solution with a non-zero imputed cost vector, the two problems share the same optimal solution. \qedwhite
\end{proof}

The key difference between Proposition~\ref{propn:gior_sol} and Corollary~\ref{cor:gior_sol_nonzero} is the non-zero assumption (i.e., $\bc^* \neq \bzero$). By relaxing the normalization constraint, we permit potential solutions for which $\bc^* = \bA^\tpose \by^* = \bzero$ is a linearly dependent combination of the rows of $\bA$.
However, if $\bc^* \neq \bzero$ is an optimal solution to the relaxed problem, it is also an optimal solution to $\GIOr{\dataset}$. Therefore, to solve $\GIOr{\dataset}$, we suggest first solving the three relaxed problems, which are LPs,
from Corollary~\ref{cor:gior_sol_nonzero}. If $\bc^* = \bzero$, then we use the more general approach. Section~\ref{sec:imrt} (with details on the formulations in~\ref{sec:ec_imrt_gio}) shows a case where the LP relaxations via Corollary~\ref{cor:gior_sol_nonzero} are sufficient.

\section{Related work in inverse convex optimization}
\label{sec:ec_inverse_convex}

Multi-point inverse optimization has recently received significant interest under the setting of convex forward problems, with several notable inverse optimization models having been proposed for arbitrary convex forward problems (i.e.,~\citet{ref:bertsimas_mp15, ref:aswani_arxiv15, ref:esfahani_oo15}). The methods proposed in this prior work specialize to linear forward problems and overlap in formulation with the absolute duality and the decision space models proposed in this paper. However, the geometric nature of LPs poses new challenges, but also allows for some efficient solutions, that are not present in the strictly convex domain. In this section, we highlight the previous formulations and discuss several differences in the solution methods.

The inverse convex models in prior work assume that the data set consists of points corresponding to different forward problem instances. As we focus on inverse optimization for a fixed forward feasible region, we illustrate the results in the previous work by fixing $\feas$.

\subsection{Inverse variational inequality}
\label{sec:ec_inverse_vi}

Let $f(\bx; \bc): \field{R}^n \rightarrow \field{R}$ be a convex function in $\bx$ parametrized by $\bc$ and $\set{K}$ be a convex cone.~\citet{ref:bertsimas_mp15} considered the forward problem $\min_\bx \left\{ f(\bx; \bc) \mid \bA \bx = \bb, \bx \in \set{K} \right\}$ and proposed an inverse optimization model that minimized the residuals from failing to satisfy the variational inequality of the first-order optimality condition. The inverse variational inequality problem is
\begin{align}
\begin{split}
    \minimize{\bc, \by_1,\dots,\by_Q, \epsilon_1,\dots,\epsilon_Q}\quad  & \sum_{q=1}^Q | \epsilon_q | \\
    \subjectto \quad    & \bA^\tpose \by_q \leq_{\set{K}} \nabla f(\bhx_q; \bc), \quad \forall q \in \set{Q} \\
                        & \nabla f(\bhx_q; \bc)^\tpose \bhx_q - \bb^\tpose \by_q \leq \epsilon_q, \quad \forall q \in \set{Q} \\
                        & \bc \in \set{C}.
\end{split} \label{eq:inverse_variational}
\end{align}
Setting $\set{K} = \field{R}^n_+$, $f(\bx; \bc) = \bc^\tpose \bx$, and $\set{C} = \left\{ \bc \in \field{R}^n \mid \nnorm{\bc} = 1 \right\}$ makes formulation~\eqref{eq:inverse_variational} equivalent to $\GIOa{\dataset}$, i.e., formulation~\eqref{eq:gioa_reform}.

In the original work,~\citet{ref:bertsimas_mp15} focused mostly on strictly convex forward problems and on ensuring a convex inverse optimization formulation. While the non-convex normalization constraint is not always necessary when the forward problem is strictly convex, setting $f(\bx; \bc) = \bc^\tpose \bx$ implies that $(\bc, \by, \epsilon_1,\dots,\epsilon_Q) = (\bzero, \bzero, 0,\dots,0)$ is a trivially optimal solution~\citep{ref:chan_gof_16, ref:esfahani_oo15}. Note furthermore that convex normalization constraints exist in the literature, e.g.,~\citet{ref:keshavarz_isic11} proposed setting $c_0 = 1$. However, these convex normalization constraints often bias the parameter space. For instance, setting $c_0 = 1$ prevents imputing non-trivial cost vectors where $c_0 = 0$. We enforce the non-convex constraint within all of the inverse optimization models in the current paper and propose polyhedral decomposition-based solution methods in the general setting for $\GIOa{\dataset}$. Furthermore, we find it important to explore special cases where the non-convexity can be bypassed, leading to simpler, sometimes analytic results (see Proposition~\ref{propn:gioa_optimal} and~\ref{propn:gioa_infeas_sol}, as well as Corollary~\ref{cor:gioa_infeas_sol_mp}).

Finally,~\citet{ref:bertsimas_mp15} discussed a decision space alternative to formulation~\eqref{eq:inverse_variational} where instead of minimizing the variational inequality residual, they minimized $\norm{\bhx_q - \bx_q}$, where $\bx_q$ is a variable that satisfies $f(\bx_q;\bc) = \bb^\tpose \by$. Furthermore, they assumed that the gradient of the objective function is strongly monotone, i.e., there exists $\gamma > 0$ such that
\begin{align*}
    {(\nabla f(\bx;\bc) - \nabla f(\by;\bc))}^\tpose(\bx - \by) \geq \gamma \norm{\bx - \by}_2, \quad \forall \bx, \by \in \feas.
\end{align*}
By focusing on the variational inequality nature of objective space inverse optimization,~\citet[Theorem 1]{ref:bertsimas_mp15} translated the variational inequality error bound of~\citet{ref:pang_1987} to show that if there exists an solution $(\bc^*, \by^*, \epsilon^*_1,\dots,\epsilon^*_Q)$ to formulation~\eqref{eq:inverse_variational}, then there exists $\bx^*_1,\dots,\bx^*_Q$ that are optimal solutions to the forward problem and satisfy $\norm{\bhx_q  - \bx^*_q}_2 \leq \sqrt{\epsilon_q/\gamma}$ for all $q$. That is, given the feasible solution to an objective space inverse optimization problem, we can obtain a corresponding feasible solution to a decision space problem where the error is bounded. Note, however, that in the linear case, $\nabla f(\bx;\bc) = \bc$ does not satisfy the strong monotone property, i.e., $\gamma = 0$. As a result, the previous bound does not hold for inverse linear optimization.

\subsection{Inverse empirical risk minimization}
\label{sec:ec_inverse_saa}

Let $f(\bx; \bu, \bc): \field{R}^n \rightarrow \field{R}$ and $g(\bx; \bu, \bc): \field{R}^n \rightarrow \field{R}^m$ be convex functions in $\bx$ that are both parametrized by $\bu$ and $\bc$.~\citet{ref:aswani_arxiv15} considered the general convex forward problem $\min_\bx \left\{ f(\bx; \bu, \bc) \mid g(\bx; \bu, \bc) \leq \bzero \right\}$ and proposed a bilevel inverse optimization model that minimized the empirical distance between a data set $\dataset = \left\{ (\bhx_1,\bhu_1),\dots, (\bhx_Q, \bhu_Q)\right\}$ of $Q$ points sampled i.i.d\@. from some joint probability distribution $\field{P}_{\bx,\bu}$ and the optimal solution sets. The corresponding inverse risk minimization problem is
\begin{align}
\begin{split}
    \minimize{\bc, \bepsilon_1,\dots,\bepsilon_Q}\quad  & \sum_{q=1}^Q \norm{\bepsilon_q}_p \\
    \subjectto \quad    & \bhx_q - \bepsilon_q \in \argmin_{\bx} \left\{ f(\bx; \bhu_q, \bc) \mid g(\bx; \bhu_q, \bc) \leq \bzero \right\}, \quad \forall q \in \set{Q} \\
                        & \bc \in \set{C}.
\end{split} \label{eq:inverse_saa}
\end{align}
Setting $f(\bx; \bu, \bc) = \bc^\tpose \bx$, $g(\bx; \bu, \bc) = \bb - \bA \bx$, and $\set{C} = \left\{ \bc \in \field{R}^n \mid \nnorm{\bc} = 1 \right\}$ specializes formulation~\eqref{eq:inverse_saa} to an equivalent form as $\GIOd{\dataset}$.

Formulation~\eqref{eq:inverse_variational} satisfies statistical consistency (i.e., given sufficient points, the imputed $\bc$ converges to a true data-generating $\bc$) under several assumptions on the data set and the forward model~\citep{ref:aswani_arxiv15}:
\begin{enumerate}
    \item \textbf{Assumption 2:} The parameter space $\set{C}$ is convex.
    
    \item \textbf{Regularity 1:} The feasible set $\feas$ is closed and bounded.
    
    \item \textbf{Identifiability condition:} There exists a unique $\bc^*$ such that:
    \begin{enumerate}
        \item The data set corresponds to noisy perturbations of optimal solutions, i.e., $\bhx_q = \bx^*_q + \bw_q$, where $\bx^*_q \in \argmin_{\bx}\left\{ f(\bx; \bu, \bc) \mid g(\bx; \bu, \bc) \leq \bzero \right\}$, and $\bw_q$ is a random variable with mean $0$ and finite variance.
        
        \item For any $\bc \neq \bc^*$, there exists $\set{U}_\bc$ such that the marginal distribution $\field{P}_{\bu}(\bu \in \set{U}_\bc) > 0$ and the optimal value for
        \begin{align*}
            \inf_{\bx, \bx^*} \quad & \norm{\bx - \bx^*} \\
            \st \quad               & \bx \in \argmin_{\bw}\left\{ f(\bw; \bu, \bc) \mid g(\bw;\bu, \bc) \leq \bzero \right\} \\
                                    & \bx^* \in \argmin_{\bw}\left\{ f(\bw; \bu, \bc^*) \mid g(\bw;\bu, \bc^*) \leq \bzero \right\}  
        \end{align*}
        is equal to $0$ for all $\bu \in \set{U}_\bc$.
        %
        
        \item For all $\bc$, 
        \begin{align*}
            \field{P}_{\bu}\left(\left\{ \bu \;\bigg|\; \left|\argmin_{\bx}\left\{ f(\bx; \bu, \bc) \mid g(\bx; \bu, \bc) \leq \bzero \right\}\right| > 1 \right\}\right) = 0
        \end{align*} 
    \end{enumerate}
\end{enumerate}
These assumptions do not hold in this work where we focus on a fixed linear forward problem for all data points. Particularly, setting $f(\bx; \bu, \bc) = \bc^\tpose \bx$ and $g(\bx; \bu, \bc) = \bb - \bA \bx$ implies that the forward and inverse optimization problem do not depend on $\bu$. Consequently, the second Identifiability condition does not hold in many settings. A trivial example is $\feas = \{ (x_1, x_2) \;|\; 0 \leq x_1, x_2 \leq 1 \}$. Here, for any cost vector $\bc^*$, there exists another cost vector $\bc_i = \ba_i/\nnorm{\ba_i}$ such that the facet described by $\bc_i$ contains an optimal vertex of $\FO{\bc^*}$. Furthermore, the third condition is also trivially violated when $\bc  = \ba_i$ for any $i \in \set{I}$. Finally, our application in Section~\ref{sec:imrt} is an example where the dataset does not correspond to noisy perturbations, but is obtained via several prediction models; we therefore cannot guarantee a well-behaved $\bw_q$. We also remark that our problem setting permits the feasible set $\feas$ to be unbounded. A last consequence of $\bu$ not existing in our setting is that the parameter space becomes non-convex due to the norm constraint. Overall, we find our problem setting to be incompatible with the statistical consistency guarantees in~\citet{ref:aswani_arxiv15}.

\citet{ref:aswani_arxiv15} propose an efficient semi-parametric algorithm to solve formulation~\eqref{eq:inverse_saa} under the assumption that the forward problem is strictly convex in $\bx$. For when $f(\bx;\bu,\bc)$ is linear however,~\citet{ref:aswani_arxiv15} introduce an enumerative algorithm for solving formulation~\eqref{eq:inverse_saa} that relies on quantizing the set $\set{C}$ to a finite set $\hat{\set{C}}$, and solving the corresponding formulation with fixed $\bc \in \hat{\set{C}}$. This algorithm is effective primarily because, for fixed $\bc$, formulation~\eqref{eq:inverse_saa} (and incidentally, $\GIOd{\dataset}$) are convex. However, the authors state that due to the enumerative nature, the algorithm is generally only applicable when the parameter space is modest (e.g., $n\leq5$ is recommended). We find that the algorithm of~\citet{ref:aswani_arxiv15} is complementary to ours. That is, their algorithm is inefficient for large $n$, while our decision space algorithm is relatively insensitive to the increase in $n$, but is inefficient for large $m$.

\subsection{Distributionally robust inverse optimization}
\label{sec:ec_inverse_rm}

\citet{ref:esfahani_oo15} study distributionally robust generalized inverse optimization for convex forward problems. Let $\varrho(\cdot)$ denote a risk measure such as the Value-at-Risk (VaR) or Conditional Value-at-Risk (CVaR). The \emph{non-robust} version of their formulation is
\begin{align}
    \begin{split}
    \minimize{\bc,\bepsilon_1,\dots,\bepsilon_Q}\quad   &\varrho(\norm{\bepsilon_1},\dots,\norm{\bepsilon_Q}) \\
    \subjectto \quad
    & \text{Constraints in~\eqref{eq:inverse_variational} or~\eqref{eq:inverse_saa}} 
    \end{split}\label{eq:inverse_risk}
\end{align}
%
\citet{ref:esfahani_oo15} consider several different variants of inverse convex optimization to encapsulate previous methods; the variants are referred to as predictability (i.e., inverse risk minimization), sub-optimality, first-order (i.e., inverse variational inequality), and bounded rationality. When the forward problem is an LP, the sub-optimality loss model is in fact equivalent to the first-order loss model, and therefore also equivalent to $\GIOa{\dataset}$ proposed here.

A consequence of the general formulation~\eqref{eq:inverse_risk} is that it leads to a new dominance relationship to bound the optimal values between predictability and sub-optimality losses. Similar to~\citet{ref:bertsimas_mp15},~\citet{ref:esfahani_oo15} define the parameter $\gamma \geq 0$ to be the largest parameter satisfying
\begin{align*}
    f(\bx;\bu,\bc) -  f(\by;\bu,\bc) \geq \nabla f(\bx;\bu,\bc)^\tpose (\bx - \by) + \frac{\gamma}{2} \norm{\bx - \by}_2^2, \quad \forall \bx,\by \in \feas, \bu \in \set{U}.
\end{align*}
Under this definition,~\citet{ref:esfahani_oo15} show that their sub-optimality (i.e., objective space) loss upper bounds their predictability (i.e., decision space) loss by a multiplicative factor $\gamma/2$. However, similar to the scenario in the previous bound, $\gamma = 0$ when $f(\bx;\bu,\bc) = \bc^\tpose \bx$. Consequently, this bound also does not hold for LP forward problems. 

\citet{ref:esfahani_oo15} focus on solving a distributionally robust version of formulation~\eqref{eq:inverse_risk}, where the robustness is over the worst-case distribution of data. As they primarily address the sub-optimality loss model, which specializes to the absolute duality gap model in this work, the comparison between their solution methods and ours is similar to the comparison between~\citet{ref:bertsimas_mp15} and ours. That is, we focus on developing efficient algorithms based on LP geometry, and as a consequence, yield several new efficiencies in the absolute duality gap setting.

\section{Automated radiation therapy treatment planning}
\label{sec:ec_imrt_gio}

IMRT treatment 
is delivered by a linear accelerator (LINAC) that 
delivers 
high-energy X-rays from different angles to a patient
's tumor. The patient's body is discretized into tiny voxels in order to calculate the dose delivered to each voxel. The design of an IMRT treatment plan is typically done by mathematical optimization where the decision variable $\bx = (\bw, \bd)$ is composed of two components that represent the beamlets and the dose delivered  (in Gy) as a result of the intensities of the beamlets, respectively.

The forward model in our experiments is a modified version of the one used by~\citet{ref:babier_2018a}. Let $\set{B}$ denote the index set of beamlets and $w_b$ be the radiation intensity of beamlet $b \in \set{B}$. Similarly, let $\set{V}$ denote the index set of voxels within a patient and $d_v$ be the dose of radiation delivered to voxel $v \in \set{V}$. Dose is calculated via a weighted linear combination of all beamlet intensities, i.e., $d_v = \sum_{b \in \set{B}} D_{v,b} w_b$, where $D_{v,b}$ is the dose influence of beamlet $b$ on voxel $v$. 

For each patient, let $\set{T}$ denote the index set of the three 
planning target volumes (PTVs) with different prescription doses (i.e., PTV56, PTV63, and PTV70 with 56 Gy, 63 Gy, and 70Gy as prescription doses, respectively) and let $\set{O}$ denote the index set of the eight surrounding OARs (i.e., brain stem, spinal cord, right parotid, left parotid, larynx, esophagus, mandible, and limPostNeck). Note that the limPostNeck is an artificially defined region used solely in optimization; it does not possess a clinical criteria. For each $t \in \set{T}$ and $o \in \set{O}$, let $\set{V}_t$ and $\set{V}_o$ denote the set of voxels corresponding to the given target or OARs, respectively. 


\subsection{Forward objectives}
\label{sec:ec_imrt_gio_fo}
The IMRT forward problem includes 65 different objectives each minimizing some feature of the dose delivered to an OAR or PTV\@. For each OAR, we minimize the mean dose delivered, the maximum dose delivered, and the average dose above a threshold $\phi^\theta_o$. Here, $\phi^\theta_o$ is a fraction $\theta$ of the average maximum dose to OAR $o$ over the data set of predictions; we consider 
$\theta \in \Theta := \left\{ 0.25, 0.5, 0.75, 0.9, 0.975 \right\}$. Such objectives for each OAR can be computed as follows:
%
%
\begin{align}
        z^{\mathrm{mean}}_o &= \frac{1}{| \set{V}_o |} \sum_{v \in \set{V}_o} d_v, \quad \forall o \in \set{O} \label{eq:imrt_obj1} \\
        z^{\mathrm{max}}_o  &= \max_{v \in \set{V}_o} \left\{ d_v \right\}, \quad \forall o \in \set{O} \label{eq:imrt_obj2} \\
        z^{\mathrm{thresh}, \theta}_o &= \frac{1}{| \set{V}_o |} \sum_{v \in \set{V}_o} \max \left\{ 0, d_v - \phi^\theta_o \right\}, \quad \forall \theta \in \Theta, \forall o \in \set{O}. \label{eq:imrt_obj3}
\end{align}
Each PTV is assigned a prescribed dose $\phi_t$, i.e., 56 Gy for PTV56, 63 Gy for PTV63, and 70 Gy for PTV70. For each PTV, we minimize the dose over the prescription, under the prescription, and the maximum dose delivered to the target, which can be computed as  follows:
\begin{align}
        z^{\mathrm{over}}_t     &= \frac{1}{| \set{V}_t |} \sum_{v \in \set{V}_t} \max \left\{ 0, d_v - \phi_t \right\}, \quad \forall t \in \set{T} \label{eq:imrt_obj4} \\
        z^{\mathrm{under}}_t    &= \frac{1}{| \set{V}_t |} \sum_{v \in \set{V}_t} \max \left\{ 0, \phi_t - d_v \right\}, \quad \forall t \in \set{T} \label{eq:imrt_obj5} \\
        z^{\mathrm{max}}_t      &= \max_{v \in \set{V}_t} \left\{ d_v \right\}, \quad \forall t \in \set{T}. \label{eq:imrt_obj6}
\end{align}

\subsection{Forward constraints}

In order to ensure that no OAR or PTV is prioritized by the objectives at a cost to the other organs, we assign a set of 
hard 
constraints 
for each structure. Every OAR is assigned a constraint to ensure that the mean dose and maximum dose do not exceed baseline safety limits. Similarly, every PTV is assigned a constraint to ensure that it receives a baseline dose on average.  

\comm{
The safety constraints are relaxations of the clinical criteria used to evaluate plans. Recall that clinical plans typically do not satisfy all of the clinical criteria. In fact, satisfying all of the criteria is infeasible for most patients. Consequently, we set these safety constraints so that all plans can satisfy at least these baseline doses for each of the OARs and PTVs; we then use the objectives to push the doses to achieving the clinical criteria. 
}
The baseline values, i.e., right-hand-side of the constraints, are obtained from the average and maximum dose delivered by the 130 clinical plans in our training set. We list the constraints below:
\begin{align}
        \text{Brain stem:}\qquad   & z^{\mathrm{mean}}_o \leq 30, \quad z^{\mathrm{max}}_o \leq 53  \\
        \text{Spinal cord:}\qquad  & z^{\mathrm{mean}}_o \leq 30, \quad z^{\mathrm{max}}_o \leq 46  \\
        \text{Left parotid:}\qquad & z^{\mathrm{mean}}_o \leq 68, \quad z^{\mathrm{max}}_o \leq 77  \\
        \text{Right parotid:}\qquad& z^{\mathrm{mean}}_o \leq 68, \quad z^{\mathrm{max}}_o \leq 78  \\
        \text{Larynx:}\qquad       & z^{\mathrm{mean}}_o \leq 68, \quad z^{\mathrm{max}}_o \leq 77  \\
        \text{Esophagus:}\qquad    & z^{\mathrm{mean}}_o \leq 52, \quad z^{\mathrm{max}}_o \leq 75  \\
        \text{Mandible:}\qquad     & z^{\mathrm{mean}}_o \leq 63, \quad z^{\mathrm{max}}_o \leq 76  \\
        \text{limPostNeck:}\qquad  & z^{\mathrm{mean}}_o \leq 21, \quad z^{\mathrm{max}}_o \leq 46  \\
        \text{PTV56:}\qquad        & z^{\mathrm{mean}}_t \geq 58  \\
        \text{PTV63:}\qquad        & z^{\mathrm{mean}}_t \geq 63  \\
        \text{PTV70:}\qquad        & z^{\mathrm{mean}}_t \geq 69
\end{align}
Note that we introduce a $z^{\mathrm{mean}}_t$ variable for the targets, analogous to  $z^{\mathrm{mean}}_o$ in~\eqref{eq:imrt_obj1}.

Finally, we include a constraint on the ``complexity'' or physical deliverability of the treatment plan. This constraint, known as the sum-of-positive-gradients (SPG), restricts the variation 
of radiation doses from neighboring beamlets so that the resulting dose shape is deliverable by the LINAC\@~\citep{ref:craft_2007}. Let $a \in \set{A}$ index each angle of the LINAC, $r \in \set{R}_a$ index each row of the LINAC at that angle, and $\set{B}_r$ be the index set of beamlets along that row. 
Then, we add the following constraint to restrict the variation of doses to be delivered from different beamlets:
%
\begin{align}
        \sum_{a \in \set{A}} \max_{r \in \set{R}_a} \left\{ \sum_{b \in \set{B}_r} \max\left\{ 0, w_b - w_{b+1} \right\} \right\} \leq 55, \label{eq:imrt_spg}
\end{align}
where we set $w_{b+1} = 0$ for 
the last beamlet in each row. The right-hand-side, i.e., the SPG, is set to 55 Gy, following the convention from previous literature~\citep{ref:babier_2018med}.

\subsection{Forward optimization problem}

The final forward problem is then to minimize a weighted combination of the objectives:
\begin{align}
        \begin{split}
                \FORT{\balpha}:\quad\minimize{\bz,\bw,\bd} \quad   & \sum_{o \in \set{O}} \left( \alpha^{\mathrm{mean}}_o z^{\mathrm{mean}}_o + \alpha^{\mathrm{max}}_o z^{\mathrm{max}}_o + \sum_{\theta \in \Theta} \alpha^{\mathrm{thresh}, \theta}_o z^{\mathrm{thresh}, \theta}_o \right) + \\
                & \qquad \sum_{t \in \set{T}} \left( \alpha^{\mathrm{over}}_t z^{\mathrm{over}}_t + \alpha^{\mathrm{under}}_t z^{\mathrm{under}}_t + \alpha^{\mathrm{max}}_t z^{\mathrm{max}}_t \right) \\
                \subjectto \quad    & \eqref{eq:imrt_obj1}-\eqref{eq:imrt_spg} \\
                & \sum_{b \in \set{B}} D_{v,b} w_b = d_v, \quad \forall v \in \set{V} \\
                & w_b, d_v \geq 0, \quad\forall b \in \set{B}, \forall v \in \set{V}.
        \end{split} \label{eq:imrt_fo}
\end{align}
We compress the notation of the above forward problem to $\FO{\balpha}: \min_{\bx} \left\{ \balpha^\tpose \bC \bx \mid \bA \bx \geq \bb, \bx \geq \bzero \right\}$. This problem has several useful properties. Firstly under this notation, the matrix of objective functions $\bC$ is non-negative. Furthermore, the constraint vector $\bb$ is also non-negative. These properties are useful specifically as they allow for constructing almost entirely linear inverse optimization problems. We discuss these in Section~\ref{sec:imrt_io_problems}.

\subsection{Generating a data set of predicted treatments}

We use the training set of $130$ patients to implement several machine learning models from the KBP literature. 
\comm{
Each model is trained via supervised learning using a data set of paired patient CT images (i.e., features) and clinically delivered dose distributions (i.e., target). There are some variations in how each model approaches the task. We use the same training techniques for each model as described in their original papers and summarize the results below:
%
\begin{enumerate}
        \item \textbf{Random Forest:} A random forest that uses $10$ hand-crafted geometric features derived from the CT images (e.g., distance to nearest tumor structure) to predict the dose for each voxel $\hat{d}_v$ of the patient individually~\citep{ref:mcintosh_2017, ref:mahmood_2018gancer}. We apply the random forest to predict each voxel for a given patient independently and concatenate the predictions to construct a dose distribution.

        \item \textbf{2-D RGB GAN:} A generative adversarial network that uses 2-D axial slices of the patient's CT as an RGB image to predict corresponding 2-D axial slices of the patient's dose also as an RGB image~\citep{ref:mahmood_2018gancer}. We convert the images to grayscale and run 2-D RGB GAN over all 128 axial slices of the patient and concatenate the predictiosn to produce a 3-D dose distribution.

        \item \textbf{2-D GANCER:} A generative adversarial network that uses 2-D axial slices of the patient's CT as an RGB image to predict 2-D axial slices of the patient's dose in grayscale directly~\citep{ref:babier_2018med}. This model is a variant of the 2-D RGB GAN. We run this model over all 128 axial slices of a patient and concatenate the predictions to produce a 3-D dose distributions.
        
        \item \textbf{3-D GANCER:} A generative adversarial network that uses the full 3-D patient's CT image as input to predict the full 3-D dose distribution $\bhd$ in one shot~\citep{ref:babier_2018med}.

\end{enumerate}
}
\begin{table}[t]
    \centering
    \caption{
        \comm{The percentage of predictions that are feasible with respect to their forward problems.}
    }
    \begin{tabular}{ c c }
    \toprule
    Predictive model    & \%-age of feasible predictions \\ \midrule
    3-D GANCER      & 95.3\\
    2-D RGB GAN     & 90.1\\
    2-D GANCER      & 82.3\\
    2-D RGB GAN-sc. & 83.9\\
    RF-sc.          & 82.3\\
    RF              & 86.2\\
    2-D GANCER-sc.  & 87.7\\
    3-D GANCER-sc.  & 86.9\\ 
    \bottomrule
    \end{tabular}
    \label{tab:fraction_feas}
\end{table}

\citet{ref:babier_2018med} noted that plans predicted using the above models often sought to deliver low dose (i.e., significantly spare healthy tissue) at the cost of not satisfying the prescription criteria for the PTVs
, and implemented a rescaling method to create a modified prediction to address this issue. In their experiments, they showed that treatment plans constructed using inverse optimization-based KBP and the normalized dose distributions would better satisfy the prescription criteria while performing slightly poorer on sparing healthy tissue. Consequently, we implement the rescaling method on all predictions from  the models, and use both the non-scaled and scaled predictions as input for the inverse optimization model. Thus, for each patient 
there is a data set of $8$ dose distributions, i.e., $\dataset = \left\{ \bhz_1,\dots,\bhz_8 \right\}$. Note that we do not require $\bhx_q = (\bhw_q, \bhd_q)$, but only the objective function values. Inverse optimization using this data set then yields a weight vector $\balpha_k$, with which we solve $\FO{\balpha_k}$ to obtain a reconstructed personalized treatment.

\comm{
Dose predictions may be feasible and sub-optimal or infeasible. Recall from Proposition~\ref{propn:gior_optimal} and~\ref{propn:gior_optimal} that if all decisions in the data set, then solving the ensemble absolute or relative duality gap inverse optimization is equivalent to solving a single-point model using the centroid. Table~\ref{tab:fraction_feas} highlights the percentage of the patients for which the predictions are feasible dose distributions with respect to $\FORT{\balpha}$. Typically we observe that about 85\% of predictions are feasible, suggesting that there is usually at least one prediction for every patient which is infeasible.
}

\subsection{Inverse optimization problems}
\label{sec:imrt_io_problems}

In order to frame $\FO{\balpha}$ for generalized inverse optimization, we restrict imputed cost vectors to be in the image of $\bC$, i.e., $\set{C} = \left\{ \bC^\tpose \balpha \mid \balpha \geq \bzero \right\}$. Note that $\balpha \geq \bzero$ is an application-specific constraint, as there is no clinical interpretation for negative objective function weights.

A specific inverse optimization problem is formulated by appropriately selecting the model hyperparameters $(\norm{\cdot}, \set{C}, \set{E}_1,\dots, \set{E}_Q)$ from $\GMIO{\dataset}$. In our experiments, we use the default parameters, except with the custom $\set{C}$ to ensure the objective function is a weighted combination of the different objectives. Moreover, we set $\nnorm{\cdot} = \norm{\cdot}_1$.

\subsubsection{Absolute duality gap.}~
Using Proposition~\ref{propn:gioa_specialization} and our specific choice of $\set{C}$, we formulate an absolute duality gap inverse optimization problem:
\begin{align}
        \begin{split}
                \IORTa{\dataset}:\quad\min_{\balpha, \by, \epsilon_1,\dots,\epsilon_Q}  & \sum_{q=1}^Q | \epsilon_q | \\
                \st \quad   & \bC^\tpose \balpha \geq \bA^\tpose \by, \quad \by \geq \bzero \\
                & \balpha^\tpose \bhz_q = \bb^\tpose \by + \epsilon_q, \quad \forall q \in \set{Q} \\
                & (\bC^\tpose\balpha)^\tpose \bone = 1  \\
                & \balpha \geq \bzero.
        \end{split}\label{eq:gioa_imrt_formulation}
\end{align}
$\IORTa{\dataset}$ is obtained by substituting $\bc = \bC^\tpose \balpha$ into formulation~\eqref{eq:gio}, and noting that $\norm{\bC^\tpose \balpha}_1 = (\bC^\tpose \balpha)^\tpose \bone$ when both $\balpha \geq \bzero$ and $\bC \geq \bzero$.

\subsubsection{Relative duality gap.}~
Using Proposition~\ref{propn:gior_specialization} and our specific choice of $\set{C}$, we formulate a relative duality gap inverse optimization problem. We then use Corollary~\ref{cor:gior_sol_nonzero} to obtain the LP relaxation of the relative duality gap problem. The two relevant formulations are given below.

\begin{minipage}{.48\linewidth}
        \begin{align}
                \begin{split}
                        &\IORTr{\dataset}:\\
                        \min_{\balpha, \by, \epsilon_1,\dots,\epsilon_Q}  & \sum_{q=1}^Q | \epsilon_q - 1 | \\
                        \st \quad   & \bC^\tpose \balpha \geq \bA^\tpose \by, \quad \by \geq \bzero \\
                        & \balpha^\tpose \bhz_q = \epsilon_q\bb^\tpose \by, \quad \forall q \in \set{Q} \\
                        & (\bC^\tpose\balpha)^\tpose \bone = 1 \\
                        & \balpha \geq \bzero. 
                \end{split}\label{eq:gior_imrt_formulation}
        \end{align}
\end{minipage}
\begin{minipage}{.48\linewidth}
        \begin{align}
                \begin{split}
                        &\IORTrlp{\dataset}:\\
                        \min_{\balpha, \by, \epsilon_1,\dots,\epsilon_Q}  & \sum_{q=1}^Q | \epsilon_q - 1 | \\
                        \st \quad   & \bC^\tpose \balpha \geq \bA^\tpose \by, \quad \by \geq \bzero \\
                        & \balpha^\tpose \bhz_q = \epsilon_q,  \quad \forall q \in \set{Q} \\
                        & \bb^\tpose \by = 1 \\
                        & \balpha \geq \bzero.
                \end{split}\label{eq:gior_imrt_formulation2}
        \end{align}
\end{minipage}

Using Algorithm~\ref{alg:gior}, we first solve the LP relaxation of $\IORTr{\dataset}$, stated above as $\IORTrlp{\dataset}$. Note that this relaxation is the application-specific analogue of $\GIOrprel{\dataset}$, which is only one of the three reformulations of the relative duality gap problem. We do not construct or solve relaxations of the other two (e.g., $\GIOrnrel{\dataset}$ and $\GIOrzrel{\dataset}$) due to the following reasons. First, the analogue to $\GIOrnrel{\dataset}$ is infeasible; in our application, $\bb \geq \bzero$ implying $\bb^\tpose \by \geq \bzero$ for all $\by \geq \bzero$. Second, the application-specific analogue of $\GIOrz{\dataset}$ in practice is often infeasible or generates plans that perform poorly on the clinical criteria satisfaction metrics compared to $\IORTrlp{\dataset}$. Recall that $\GIOrz{\dataset}$ requires $\bc^\tpose \bhx_q = 0$ for all $q \in \set{Q}$. In the application-specific analogue (where the constraint is $\balpha^\tpose \bhz_q = 0$), both $\balpha \geq 0$ and $\bhz_q \geq 0$, which means that the problem is feasible only when there exists an element of $\bhz_q$ that is equal to $0$ for all of the predictions. The only objectives where this situation could occur are the threshold objectives~\eqref{eq:imrt_obj3}--\eqref{eq:imrt_obj5}. Thus, the application-specific analogue of $\GIOrz{\dataset}$ is either infeasible or distributes all of the objective weights to these three objectives. By strictly focusing on the threshold objectives however, the inverse problem then generally fails to meet a large number of the clinical criteria. Consequently, we advocate in this application to strictly use $\IORTrlp\dataset$ to solve the relative duality gap inverse optimization problem.

\begin{algorithm}[t]
    \caption{
        \comm{Multiplicative Weights Algorithm Baseline}
    }
    \label{alg:mwa}
    \begin{algorithmic}[1]
        \renewcommand{\algorithmicrequire}{\textbf{Input:}}
        \renewcommand{\algorithmicensure}{\textbf{Output:}}
        \comm{
        \REQUIRE Data set of CT images for training patients $\set{C}$, Data set of CT images for testing patients $\tilde{\set{C}}$, Dose prediction models $F_1(\cdot), \dots, F_Q(\cdot)$, Learning rate $\eta \leq 0.5$.
        \ENSURE  Treatment plans for each patient
        \STATE Initialize weights $w_q = 1$ for $q \in \set{Q}$.
        \FOR{Each patient in the training data set $\bhc_k \in \tilde{\set{C}}$}
            \FOR{$q \in \set{Q}$}
                \STATE Let $\hat{\bd}_{q, k} \gets F_q(\bhc_k)$.
                \STATE Let $z_{q, k} \gets \IORTr{\{ \hat{\bd}_{q, k} \}}$.
                \STATE Let $w_q \gets w_q ( 1 - \eta z_{q, k} )$.
            \ENDFOR
        \ENDFOR
        \STATE Normalize weights $w_q \gets w_q / (\sum_{q'=1}^Q w_{q'})$.
        \FOR{Each patient in the testing data set $\bhc_k \in \set{C}$}
            \STATE Select prediction model $F_q(\cdot)$ with probability $w_q$.
            \STATE Let $\hat{\bd}_{q, k} \gets F_q(\bhc_k)$.
            \STATE Let $\balpha^*_k$ be the optimal solution to $\IORTr{\{ \hat{\bd}_{q, k} \}}$.
            \STATE Let $\bx^*_k \gets \FORT{\balpha^*_k}$ and evaluate the corresponding treatment plan.
        \ENDFOR
        }
    \end{algorithmic}
\end{algorithm}

\comm{
\subsection{Baseline implementations}
\label{sec:ec_baselines}

In Section~\ref{sec:imrt_comparison_vs_baselines}, we implement two conventional ensemble learning baselines to compare with ensemble inverse optimization. The first baseline is an ensemble-then-inverse optimization model. Here, we first compute the average of the individual decisions and then solve a single-point inverse optimization problem to obtain a cost vector. The second baseline is a Multiplicative Weights Algorithm (MWA). In our experiments, we implemented both models using all eight predictions as well as for the 4 Pts\@. predictions (RF-sc., RF, 2-D GANCER-sc., 3-D GANCER-sc.). We also use grid search with the training set patients to identify the best learning rate for the MWA.

Algorithm~\ref{alg:mwa} summarizes the steps for the MWA. We use an offline learning variant of the Weighted Majority update rule of~\citet{ref:arora_2012multiplicative}.
Each of the prediction models $F_q(\cdot)$ in the ensemble KBP pipeline is treated as an expert and we initialize a weight $w_q =1$ for each model. For each of the 130 training set patients $k$ and each prediction model, we predict a dose $\hat{d}_{q, k}$, solve a single-point inverse optimization problem and update the weight $w_q$ by a penalty factor corresponding to the aggregate error of the inverse optimization problem. After repeating this process for all of the training set patients, we normalize the weights to a probability distribution and freeze them. Then for each of the patients in the test set, we randomly select an `expert' KBP pipeline to generate a treatment plan. 
}

\end{document}